\documentclass[a4paper,reqno,11pt]{amsart}
\usepackage[tmargin=30mm,bmargin=30mm,hmargin=25mm]{geometry}
\usepackage{amssymb}
\usepackage{amstext}
\usepackage{amsmath}
\usepackage{amscd}
\usepackage{amsthm}
\usepackage{amsfonts}
\usepackage{enumerate}
\usepackage{graphicx}
\usepackage{latexsym}
\usepackage{mathrsfs}
\input xy
\xyoption{all}
\numberwithin{equation}{section}
\usepackage{hyperref}
\usepackage{lscape}
\usepackage{tikz}
\usetikzlibrary{arrows,decorations.pathmorphing,decorations.pathreplacing}
\tikzset{black/.style={circle,fill=black,inner sep=3pt,outer sep=3pt},white/.style={circle,fill=white,draw=black,inner sep=3pt,outer sep=3pt}}
\usetikzlibrary{cd}
\newtheorem{theorem}{Theorem}[section]

\newtheorem{corollary}[theorem]{Corollary}
\newtheorem{lemma}[theorem]{Lemma}
\newtheorem{proposition}[theorem]{Proposition}
\theoremstyle{definition}
\newtheorem{definition}[theorem]{Definition}
\newtheorem{remark}[theorem]{Remark}
\newtheorem{example}[theorem]{Example}
\newtheorem{condition}[theorem]{Condition}
\DeclareMathOperator{\cotstr}{\mathsf{co-t-str}}
\DeclareMathOperator{\bddcotstr}{\mathsf{bdd-co-t-str}}
\DeclareMathOperator{\silt}{\mathsf{silt}}

\DeclareMathOperator{\bddhcotors}{\mathsf{bdd-hcotors}}
\DeclareMathOperator{\cotors}{\mathsf{cotors}}
\DeclareMathOperator{\scotors}{\mathsf{scotors}}
\DeclareMathOperator{\hcotors}{\mathsf{hcotors}}
\newcommand{\add}{\mathsf{add}\hspace{.01in}}
\newcommand{\inj}{\mathsf{inj}\hspace{.01in}}
\renewcommand{\mod}{\mathsf{mod}\hspace{.01in}}
\newcommand{\proj}{\mathsf{proj}\hspace{.01in}}
\newcommand{\thick}{\mathsf{thick}\hspace{.01in}}
\newcommand{\cocone}{\operatorname{Cocone}\nolimits}
\newcommand{\cone}{\operatorname{Cone}\nolimits}

\newcommand{\Ext}{\operatorname{Ext}\nolimits}

\newcommand{\id}{\operatorname{id}\nolimits}

\newcommand{\op}{\operatorname{op}\nolimits}
\newcommand{\pd}{\operatorname{pd}\nolimits}

\newcommand{\htt}{\wedge}
\newcommand{\chk}{\vee}
\newcommand{\tld}{\sim}

\begin{document}
\title[Hereditary cotorsion pairs and silting subcategories]{Hereditary cotorsion pairs and silting subcategories in extriangulated categories}

\author{Takahide Adachi}
\address{T.~Adachi: Faculty of Global and Science Studies, Yamaguchi University, 1677-1 Yoshida, Yamaguchi 753-8541, Japan}
\email{tadachi@yamaguchi-u.ac.jp}
\thanks{T.~Adachi is supported by JSPS KAKENHI Grant Number JP20K14291.}

\author{Mayu Tsukamoto}
\address{M.~Tsukamoto: Graduate school of Sciences and Technology for Innovation, Yamaguchi University, 1677-1 Yoshida, Yamaguchi 753-8512, Japan}
\email{tsukamot@yamaguchi-u.ac.jp}
\thanks{M.~Tsukamoto is supported by JSPS KAKENHI Grant Number JP19K14513.}

\subjclass[2020]{18G80}
\keywords{extriangulated categories, cotorsion pairs, silting subcategories, co-$t$-structures}

\begin{abstract}
In this paper, we study (complete) cotorsion pairs in extriangulated categories. 
First, we study a relationship between an interval of the poset of cotorsion pairs and the poset of cotorsion pairs in the coheart associated to the interval.
Secondly, we establish a bijection between bounded hereditary cotorsion pairs and silting subcategories in extriangulated categories.
\end{abstract}
\maketitle

\section{Introduction}
The concept of cotorsion pairs was invented by Salce (\cite{S79}) in the category of abelian groups, and then was defined in an exact category or a triangulated category. 
In the representation theory of algebras, (complete) cotorsion pairs play a crucial role, e.g, \cite{AB89, AR91, R91, ET01}.
Recently, Nakaoka and Palu (\cite{NP19}) formalized the notion of extriangulated categories as a simultaneous generalization of exact categories and triangulated categories. 
Moreover, they introduced cotorsion pairs in an extriangulated category. 

Our first aim is to study a relationship between an interval of cotorsion pairs and the poset of cotorsion pairs in the corresponding coheart. 
Let $x_{1}:=(\mathcal{X}_{1}, \mathcal{Y}_{1})$ and $x_{2}:=(\mathcal{X}_{2}, \mathcal{Y}_{2})$ be cotorsion pairs in an extriangulated category.
We define $x_{1} \leq x_{2}$ if it satisfies $\mathcal{Y}_{1} \subseteq \mathcal{Y}_{2}$, and in this case, $[x_{1},x_{2}]$ denotes the \emph{interval} in the poset of cotorsion pairs consisting of $x$ with $x_{1} \leq x \leq x_{2}$. 
We call the subcategory $\mathcal{H}_{[x_{1}, x_{2}]}:=\mathcal{X}_{1} \cap \mathcal{Y}_{2}$ the \emph{coheart} of $[x_{1}, x_{2}]$. 
Then each coheart can be naturally regarded as an extriangulated category. 
In this setting, if $x_{1}$ and $x_{2}$ are $s$-cotorsion pairs (see Definition \ref{def_cotors}(2)), then we have the first main theorem of this paper. 

\begin{theorem}[Theorem \ref{thm2}]\label{mthm2}
Let $\mathcal{C}$ be an extriangulated category and $x_{1}, x_{2}$ $s$-cotorsion pairs.
Let $[x_{1}, x_{2}]$ be an interval in the poset of cotorsion pairs in $\mathcal{C}$ and $\mathcal{H}_{[x_{1}, x_{2}]}$ its coheart.
Then there exists an isomorphism of posets between $[x_{1}, x_{2}]$ and the poset of cotorsion pairs in $\mathcal{H}_{[x_{1}, x_{2}]}$.
\end{theorem}

By Theorem \ref{mthm2}, we obtain a bijection (\cite[Theorem 2.1]{PZ}) between intermediate co-$t$-structures and cotorsion pairs in the corresponding coheart (see Corollary \ref{cor_PZ}). 

Our second aim is to study a connection between hereditary cotorsion pairs and silting subcategories. 
Hereditary cotorsion pairs (see Definition \ref{def_cotors}(3)) are a generalization of co-$t$-structures on a triangulated category. 
The notion of co-$t$-structures was independently introduced by Bondarko (\cite{Bo10}) and Pauksztello (\cite{Pa08}) as an analog of $t$-structures defined in \cite{BBD81}.
On the other hand, the notion of silting subcategories was introduced in \cite{KV88} to study bounded $t$-structures. 
Subsequently, Koenig and Yang (\cite{KY14}) gave a bijection between algebraic $t$-structures and silting subcategories for finite dimensional algebras.
As a counterpart of this bijection, Bondarko (\cite{Bo10}), and Mendoza, Santiago, S\'aenz and Souto (\cite{MSSS13}) gave a bijection between bounded co-$t$-structures and silting subcategories on a triangulated category.
To give a generalization of their result, we introduce the notion of silting subcategories in an extriangulated categories (see Definition \ref{def_silting}).
The following theorem is our second main result of this paper.

\begin{theorem}[Theorem \ref{thm1}]\label{mthm1}
Let $\mathcal{C}$ be an extriangulated category. 
Then there exists a bijection between the set of bounded hereditary cotorsion pairs in $\mathcal{C}$ and the set of silting subcategories of $\mathcal{C}$.  
\end{theorem}

Theorem \ref{mthm1} is not only a generalization of \cite[Corollary 5.9]{MSSS13}, but it also recovers Auslander--Reiten's result (\cite[Corollary 5.6]{AR91}). 
Namely, by Theorem \ref{mthm1}, we have a bijection between  
contravariantly finite resolving subcategories and basic tilting modules for an artin algebra with finite global dimension (see Corollary \ref{cor_AR}).

\section{Preliminaries}

Throughout this paper, we fix a commutative unital ring $R$ and let $\mathcal{C}$ denote a small additive $R$-linear category. 
All subcategories are assumed to be full, additive and closed under isomorphisms. 

In this section, we collect terminologies and basic properties of extriangulated categories which we need later. 
We omit the precise definition of extriangulated categories.
For details, we refer to \cite{NP19} and \cite{INP}.

An extriangulated category $\mathcal{C}=(\mathcal{C},\mathbb{E},\mathfrak{s})$ consists of the following data which satisfy certain axioms (see \cite[Definition 2.12]{NP19}):
\begin{itemize}
\item $\mathcal{C}$ is an additive category.
\item $\mathbb{E}:\mathcal{C}^{\mathrm{op}}\times \mathcal{C}\to \mathsf{Mod}\hspace{.01in} R$ is an $R$-bilinear functor.
\item $\mathfrak{s}$ is a correspondence which associates an equivalence class $[A\rightarrow B\rightarrow C]$ of complexes in $\mathcal{C}$ to each $\delta\in \mathbb{E}(C,A)$.
Here two complexes $A\xrightarrow{f}B\xrightarrow{g}C$ and $A\xrightarrow{f'}B'\xrightarrow{g'}C$ in $\mathcal{C}$ are \emph{equivalent} if there is an isomorphism $b:B\to B'$ such that the diagram
\[
\begin{tikzcd}
A \rar["f"] \dar[equal] & B \rar["g"] \dar["b"', "\cong"] & C \dar[equal] \\
A \rar["f'"] & B' \rar["g'"] & C
\end{tikzcd}
\]
is commutative, and let $[A\xrightarrow{f}B\xrightarrow{g}C]$ denote the equivalence class of $A\xrightarrow{f}B\xrightarrow{g}C$.
\end{itemize}
A complex $A\xrightarrow{f}B\xrightarrow{g}C$ in $\mathcal{C}$ is called an \emph{$\mathfrak{s}$-conflation} if there exists $\delta\in \mathbb{E}(C,A)$ such that $\mathfrak{s}(\delta)=[A\xrightarrow{f}B\xrightarrow{g}C]$.
We write the $\mathfrak{s}$-conflation as $A\xrightarrow{f}B\xrightarrow{g}C\overset{\delta}{\dashrightarrow}$.

Recall the axiom (ET4) in extriangulated categories and \cite[Proposition 3.15]{NP19}, which are frequently used in this paper.
By (ET4), for given two $\mathfrak{s}$-conflations $A\to B\xrightarrow{b} D\dashrightarrow$ and $B\to C\to F\overset{\delta}{\dashrightarrow}$, we have two $\mathfrak{s}$-conflations $A\to C\to E\dashrightarrow$ and $D\to E\to F\overset{b_{\ast}\delta}{\dashrightarrow}$.

\begin{lemma}[{\cite[Proposition 3.15]{NP19}}] \label{lem_pb}
Let $A \rightarrow B \rightarrow F \dashrightarrow$ and $C \rightarrow D \rightarrow F \dashrightarrow$ be $\mathfrak{s}$-conflations in $\mathcal{C}$. 
Then we have two $\mathfrak{s}$-conflations $A \to E \to D \dashrightarrow$ and $C \to E \to B \dashrightarrow$.
\end{lemma}

Due to (ET4), we have the following result.

\begin{lemma}\label{lem_et4}
Let $A\to B\xrightarrow{b} D\dashrightarrow$, $B\to C\to F\overset{\delta}{\dashrightarrow}$ and $D\to E\to F\overset{\eta}{\dashrightarrow}$ be $\mathfrak{s}$-conflations satisfying $\eta=b_{\ast}\delta$.
Then we have an $\mathfrak{s}$-conflation $A\to C\to E\dashrightarrow$.
\end{lemma}

\begin{proof}
Applying (ET4) to the $\mathfrak{s}$-conflations $A\to B\xrightarrow{b} D\dashrightarrow$ and $B\to C\to F\overset{\delta}{\dashrightarrow}$, we have $\mathfrak{s}$-conflations $A\xrightarrow{a} C\xrightarrow{c} E'\overset{\theta}{\dasharrow}$ and $D\to E'\to F\overset{b_{\ast}\delta}{\dashrightarrow}$. By $\eta=b_{\ast}\delta$, we have an isomorphism $\varphi: E\to E'$. Thus, by \cite[Proposition 3.7]{NP19}, we obtain an $\mathfrak{s}$-conflation $A\xrightarrow{a}C\xrightarrow{\varphi^{-1}c}E\overset{\varphi^{\ast}\theta}{\dasharrow}$. 
\end{proof}

Gorsky, Nakaoka and Palu (\cite{GNP}) gave an $R$-bilinear functor $\mathbb{E}^{n}: \mathcal{C}^{\op} \times \mathcal{C} \to \mathsf{Mod}\hspace{.01in} R$ and proved that any $\mathfrak{s}$-conflation induces the following long exact sequences.

\begin{proposition}[{\cite[Theorem 3.5]{GNP}}]\label{prop_longex}
Let $A\xrightarrow{f}B\xrightarrow{g}C\overset{\delta}{\dashrightarrow}$ be an $\mathfrak{s}$-conflation. 
Then the following statements hold. 
\begin{itemize}
\item[(1)] For each $X\in\mathcal{C}$, there exists a long exact sequence 
\begin{align}
&\mathcal{C}(X,A) \to \mathcal{C}(X,B) \to \mathcal{C}(X,C) \to \mathbb{E}(X,A) \to \cdots \notag  \\ 
&\cdots \to \mathbb{E}^{n-1}(X,C) \to \mathbb{E}^{n}(X,A) \to \mathbb{E}^{n}(X,B) \to \mathbb{E}^{n}(X,C) \to \cdots.  \notag
\end{align}
\item[(2)] For each $X\in \mathcal{C}$, there exists a long exact sequence 
\begin{align}
&\mathcal{C}(C, X) \to \mathcal{C}(B,X) \to \mathcal{C}(A,X) \to \mathbb{E}(C,X) \to \cdots \notag \\ 
&\cdots \to \mathbb{E}^{n-1}(A,X) \to \mathbb{E}^{n}(C,X) \to \mathbb{E}^{n}(B,X) \to \mathbb{E}^{n}(A,X) \to \cdots.  \notag
\end{align}
\end{itemize}
\end{proposition}

We give two remarks on positive extensions of extriangulated subcategories.
If $\mathcal{C}$ has enough projective objects and enough injective objects, then the bilinear functor $\mathbb{E}^{n}$ is isomorphic to that in \cite{HLN21} or \cite{LN19} (see \cite[Cororally 3.21]{GNP}).
Let $\mathcal{D}$ be a subcategory with restricted extriangulated structure $(\mathcal{D},\mathbb{E}_{\mathcal{D}},\mathfrak{s}_{\mathcal{D}})$. Then for any $X,Y\in\mathcal{D}$, it satisfies $\mathbb{E}_{\mathcal{D}}(X,Y)\cong \mathbb{E}(X,Y)$, but it does not necessarily satisfy $\mathbb{E}_{\mathcal{D}}^{n}(X,Y)\cong \mathbb{E}^{n}(X,Y)$ for $n\geq 2$ (see \cite[Remark 3.29]{GNP}). 

For a subcategory $\mathcal{X}$ of $\mathcal{C}$, we define two subcategories  ${}^{\perp_{1}}\mathcal{X}$ and ${}^{\perp}\mathcal{X}$ as
\begin{align}
&{}^{\perp_{1}}\mathcal{X}:=\{ M\in \mathcal{C} \mid \mathbb{E}(M,\mathcal{X})=0\},\notag\\
&{}^{\perp}\mathcal{X}:=\{ M\in \mathcal{C} \mid \mathbb{E}^{k}(M,\mathcal{X})=0\textnormal{\;for\;each\;} k \ge 1\}.\notag
\end{align}
Dually, we define subcategories $\mathcal{X}^{\perp_{1}}$ and $\mathcal{X}^{\perp}$.

Throughout this paper, the following subcategories play a crucial role. 

\begin{definition}
Let $\mathcal{X}, \mathcal{Y}$ be subcategories of $\mathcal{C}$. 
\begin{itemize}
\item[(1)] Let $\mathcal{X} \ast \mathcal{Y}$ denote the subcategory of $\mathcal{C}$ consisting of $M\in\mathcal{C}$ which admits an $\mathfrak{s}$-conflation $X \rightarrow M \rightarrow Y \dashrightarrow$ in $\mathcal{C}$ with $X\in \mathcal{X}$ and $Y\in\mathcal{Y}$. We say that \emph{$\mathcal{X}$ is closed under extensions} if $\mathcal{X}\ast\mathcal{X}\subseteq \mathcal{X}$.
\item[(2)] Let $\cone(\mathcal{X},\mathcal{Y})$ denote the subcategory of $\mathcal{C}$ consisting of $M\in\mathcal{C}$ which admits an $\mathfrak{s}$-conflation $X \rightarrow Y \rightarrow M \dashrightarrow$ in $\mathcal{C}$ with $X\in \mathcal{X}$ and $Y\in\mathcal{Y}$.  We say that \emph{$\mathcal{X}$ is closed under cones} if $\cone(\mathcal{X},\mathcal{X})\subseteq\mathcal{X}$.
\item[(3)] Let $\cocone(\mathcal{X}, \mathcal{Y})$ denote the subcategory of $\mathcal{C}$ consisting of $M\in\mathcal{C}$ which admits an $\mathfrak{s}$-conflation $M\rightarrow X \rightarrow Y\dashrightarrow$ in $\mathcal{C}$ with $X\in \mathcal{X}$ and $Y\in\mathcal{Y}$. We say that \emph{$\mathcal{X}$ is closed under cocones} if $\cocone(\mathcal{X},\mathcal{X})\subseteq\mathcal{X}$.
\item[(4)] We call $\mathcal{X}$ a \emph{thick subcategory} of $\mathcal{C}$ if it is closed under extensions, cones, cocones and direct summands. Let $\thick \mathcal{X}$ denote the smallest thick subcategory containing $\mathcal{X}$. 
\end{itemize}
\end{definition}

The axiom (ET4), Lemmas \ref{lem_pb} and \ref{lem_et4} induce the following properties of $\mathcal{X} \ast \mathcal{Y}$, $\cone (\mathcal{X}, \mathcal{Y})$ and  $\cocone(\mathcal{X},\mathcal{Y})$.

\begin{lemma}\label{lem_basic}
For subcategories $\mathcal{X}, \mathcal{Y}, \mathcal{Z}$ of $\mathcal{C}$,  the following statements hold. 
\begin{itemize}
\item[(1)] $\cone (\mathcal{X}, \cone (\mathcal{Y}, \mathcal{Z})) \subseteq \cone (\mathcal{Y} \ast \mathcal{X}, \mathcal{Z})$. 
\item[(2)] $\cocone (\cocone (\mathcal{X}, \mathcal{Y}), \mathcal{Z}) \subseteq \cocone (\mathcal{X}, \mathcal{Z} \ast \mathcal{Y})$. 
\item[(3)] $\cone (\cocone (\mathcal{X}, \mathcal{Y}), \mathcal{Z}) )\subseteq \cone (\mathcal{X}, \mathcal{Z} \ast \mathcal{Y})$. 
\item[(4)] $\cocone (\mathcal{X}, \cone (\mathcal{Y}, \mathcal{Z}) )\subseteq \cocone (\mathcal{Y} \ast \mathcal{X}, \mathcal{Z})$. 
\item[(5)] $\mathcal{X}  \ast \cone (\mathcal{Y}, \mathcal{Z}) \subseteq \cone (\mathcal{Y}, \mathcal{X} \ast \mathcal{Z})$. 
\item[(6)] $\cocone (\mathcal{X}, \mathcal{Y}) \ast \mathcal{Z} \subseteq \cocone (\mathcal{X} \ast \mathcal{Z}, \mathcal{Y})$. 
\item[(7)] $\cone (\mathcal{X}, \cocone (\mathcal{Y}, \mathcal{Z}))= \cocone (\cone (\mathcal{X}, \mathcal{Y}), \mathcal{Z}))$. 
\item[(8)] If $\mathbb{E}^{2}(\mathcal{Z},\mathcal{X})=0$, then $\cone(\mathcal{X},\mathcal{Y})\ast\mathcal{Z}\subseteq \cone(\mathcal{X},\mathcal{Y}\ast\mathcal{Z})$.
\item[(9)] If $\mathbb{E}^{2}(\mathcal{Z},\mathcal{X})=0$, then $\mathcal{X}\ast\cocone(\mathcal{Y},\mathcal{Z})\subseteq\cocone(\mathcal{X\ast\mathcal{Y},\mathcal{Z}})$.
\end{itemize}
\end{lemma}
\begin{proof}
(1) This follows from (ET4)$^{\op}$.

(2) This follows from (ET4).

(3) This follows from the dual statement of Lemma \ref{lem_pb}.

(4) This follows from Lemma \ref{lem_pb}.

(5) This follows from Lemma \ref{lem_pb}.

(6) This follows from the dual statement of Lemma \ref{lem_pb}.

(7) This follows from (ET4) and (ET4)$^{\op}$.

(8) Let $E\in \cone(\mathcal{X},\mathcal{Y})\ast\mathcal{Z}$. Then there exist $\mathfrak{s}$-conflations $D\to E\to F\overset{\eta}{\dasharrow}$ and $A\to B\xrightarrow{b} D\dasharrow$ such that $F\in \mathcal{Z}$, $A\in \mathcal{X}$ and $B\in \mathcal{Y}$. Applying $\mathcal{C}(F,-)$ to the second $\mathfrak{s}$-conflation gives an exact sequence
\begin{align}
\mathbb{E}(F,B)\to\mathbb{E}(F,D)\to\mathbb{E}^{2}(F,A)=0,\notag
\end{align}
where the last equality follows from $\mathbb{E}^{2}(\mathcal{Z},\mathcal{X})=0$.
For $\eta\in\mathbb{E}(F,D)$, there exists $\delta\in\mathbb{E}(F,B)$ such that $\eta=b_{\ast}\delta$. Let $B\to C\to F\overset{\delta}\dashrightarrow$ be the $\mathfrak{s}$-conflation associated to $\delta$. Then $C\in\mathcal{Y}\ast\mathcal{Z}$ holds.
By Lemma \ref{lem_et4}, we have an $\mathfrak{s}$-conflation $A\to C\to E\dasharrow$.
This implies $E\in\cone(\mathcal{X},\mathcal{Y}\ast\mathcal{Z})$.

(9) It is similar to (8). 
\end{proof}

\section{Bijection between cotorsion pairs}

In this section, we recall the definition of cotorsion pairs and study a relationship between an interval of cotorsion pairs and cotorsion pairs in the coheart associated to the interval, which is a generalization of \cite{PZ, LZ} and an analog of \cite{AET}.
Let $\mathcal{C}=(\mathcal{C},\mathbb{E},\mathfrak{s})$ be an extriangulated category.
We start this section with recalling the definition of cotorsion pairs in $\mathcal{C}$. 

\begin{definition}\label{def_cotors}
Let $\mathcal{X},\mathcal{Y}$ be subcategories of $\mathcal{C}$.
\begin{itemize}
\item[(1)] We call a pair $(\mathcal{X},\mathcal{Y})$  a \emph{cotorsion pair} in $\mathcal{C}$ if it satisfies the following conditions.
\begin{itemize}
\setlength{\itemindent}{20pt}
\item[(CP1)] $\mathcal{X}$ and $\mathcal{Y}$ are closed under direct summands.
\item[(CP2)] $\mathbb{E}(\mathcal{X}, \mathcal{Y})=0$.
\item[(CP3)] $\mathcal{C}=\cone (\mathcal{Y}, \mathcal{X})$.
\item[(CP4)] $\mathcal{C}=\cocone (\mathcal{Y}, \mathcal{X})$.
\end{itemize}
\item[(2)] A cotorsion pair $(\mathcal{X}, \mathcal{Y})$ is called an \emph{$s$-cotorsion pair} if it satisfies the following condition. 
\begin{itemize}
\setlength{\itemindent}{20pt}
\item[(SCP)] $\mathbb{E}^{2}(\mathcal{X}, \mathcal{Y})=0$. 
\end{itemize}
\item[(3)] A cotorsion pair $(\mathcal{X}, \mathcal{Y})$ is called a \emph{hereditary cotorsion pair} if it satisfies the following condition.
\begin{itemize}\setlength{\itemindent}{20pt}
\item[(HCP)]$\mathbb{E}^{k}(\mathcal{X}, \mathcal{Y})=0$ for each $k \ge 2$.
\end{itemize}
\end{itemize}
\end{definition}

By \cite[Lemma 4.3]{LZ}, if $\mathcal{C}$ has enough projective/injective objects, then $s$-cotorsion pairs are hereditary cotorsion pairs. 
The following lemma tells us that if $(\mathcal{X}, \mathcal{Y})$ is an $s$-cotorsion pair, then the subcategory $\mathcal{X}$ is closed under extensions and cocones. 

\begin{lemma} \label{lem_cotors}
Let $(\mathcal{X}, \mathcal{Y})$ be a cotorsion pair in $\mathcal{C}$. 
Then the following statements hold. 
\begin{itemize}
\item[(1)] {{\cite[Remark 4.4]{NP19}}} $\mathcal{X}={}^{\perp_{1}}\mathcal{Y}$ and $\mathcal{Y}=\mathcal{X}^{\perp_{1}}$. 
In particular, $\mathcal{X}$ and $\mathcal{Y}$ are closed under extensions. 
\item[(2)] Assume that $(\mathcal{X}, \mathcal{Y})$ is an $s$-cotorsion pair. 
Then $\mathcal{X}$ is closed under cocones and $\mathcal{Y}$ is closed under cones. 
\end{itemize}
\end{lemma}

\begin{proof}
We prove (2).
Let $(\mathcal{X},\mathcal{Y})$ be an $s$-cotorsion pair.
We show that $\mathcal{X}$ is closed under cocones. 
Let $L \to M \to N \dashrightarrow$ be an $\mathfrak{s}$-conflation with $M, N \in \mathcal{X}$. 
Applying $\mathcal{C}(-, \mathcal{Y})$ to the $\mathfrak{s}$-conflation gives an exact sequence 
\begin{align}
\mathbb{E}(M, \mathcal{Y}) \to \mathbb{E}(L, \mathcal{Y}) \to \mathbb{E}^{2}(N, \mathcal{Y}).   \notag
\end{align}
Since the left-hand side and the right-hand side vanish by $M, N \in \mathcal{X}$, we have $L\in {}^{\perp_{1}}\mathcal{Y}$.
Hence the assertion follows from (1). 
Similarly, $\mathcal{Y}$ is closed under cones. 
\end{proof}

In triangulated categories, the condition (SCP) corresponds to ``shift-closed'' condition.  

\begin{lemma}\label{lem_shift}
Let $\mathcal{D}$ be a triangulated category (regarded as an extriangulated category) with shift functor $\Sigma$.
Let $(\mathcal{X},\mathcal{Y})$ be a cotorsion pair in $\mathcal{D}$.
Then the following statements are equivalent.
\begin{itemize}
\item[(1)] $(\mathcal{X},\mathcal{Y})$ satisfies the condition \textnormal{(HCP)}.
\item[(2)] $(\mathcal{X},\mathcal{Y})$ satisfies the condition \textnormal{(SCP)}.
\item[(3)] $\mathcal{X}$ is closed under negative shifts, that is, $\Sigma^{-1} \mathcal{X}\subseteq \mathcal{X}$.
\item[(4)] $\mathcal{Y}$ is closed under positive shifts, that is, $\Sigma \mathcal{Y}\subseteq\mathcal{Y}$.
\end{itemize}
\end{lemma}

\begin{proof}
For each $k \ge2$, it follows from \cite[Corollary 3.23]{GNP} that 
\begin{align}
\mathbb{E}^{k}(\mathcal{X}, \mathcal{Y})\cong \mathbb{E}(\Sigma^{-k+1}\mathcal{X}, \mathcal{Y}).     \notag
\end{align}

(1)$\Rightarrow$(2): This is clear. 

(2)$\Rightarrow$(3): Since $\mathbb{E}^{2}(\mathcal{X},\mathcal{Y})=0$, we have $\Sigma^{-1}\mathcal{X} \subseteq {}^{\perp_{1}}\mathcal{Y}$. 
By Lemma \ref{lem_cotors}(1), the assertion holds. 

(3)$\Rightarrow$(1): Since $\Sigma^{-k+1} \mathcal{X}\subseteq \mathcal{X}$ for each $k\ge2$, we have the assertion. 

Similarly, we obtain (2)$\Rightarrow$(4)$\Rightarrow$(1). 
Hence the proof is complete. 
\end{proof}

The following examples show that hereditary cotorsion pairs are a common generalization of co-$t$-structures on a triangulated category and complete hereditary cotorsion pairs in an exact category. 

\begin{example}\label{ex_scotors}
\begin{itemize}
\item[(1)] Let $\mathcal{D}$ be a triangulated category with shift functor $\Sigma$. 
A pair $(\mathcal{U},\mathcal{V})$ of subcategories of $\mathcal{D}$ is called a \emph{co-$t$-structure} on $\mathcal{D}$ if it satisfies the following conditions. 
\begin{itemize}
\item[$\bullet$] $\mathcal{U}$ and $\mathcal{V}$ are closed under direct summands.
\item[$\bullet$] $\mathcal{D}=\Sigma^{-1}\mathcal{U}\ast\mathcal{V}$, that is, for each $D\in\mathcal{D}$, there exists a triangle $\Sigma^{-1}U\to D\to V\to U$ such that $U\in \mathcal{U}$ and $V\in\mathcal{V}$.
\item[$\bullet$] $\mathcal{D}(\Sigma^{-1}\mathcal{U},\mathcal{V})=0$.
\item[$\bullet$] $\mathcal{U}$ is closed under negative shifts. \end{itemize}
By regarding $\mathcal{D}$ as an extriangulated category, it follows from Lemma \ref{lem_shift} that co-$t$-structures on $\mathcal{D}$ are exactly hereditary cotorsion pairs. 
\item[(2)] Let $\mathcal{E}$ be an exact category. A pair $(\mathcal{X},\mathcal{Y})$ of subcategories of $\mathcal{E}$ is called a \emph{complete hereditary cotorsion pair} in $\mathcal{E}$ if it satisfies the following conditions.
\begin{itemize}
\item[$\bullet$] $\mathcal{X}$ and $\mathcal{Y}$ are closed under direct summands.
\item[$\bullet$] $\Ext_{\mathcal{E}}^{k}(\mathcal{X},\mathcal{Y})=0$ for each $k \ge 1$.
\item[$\bullet$] For each $E\in \mathcal{E}$, there exists a conflation $0\to Y_{E}\to X_{E}\to E\to 0$ such that $Y_{E}\in\mathcal{Y}$ and $X_{E}\in\mathcal{X}$.
\item[$\bullet$] For each $E\in \mathcal{E}$, there exists a conflation $0 \to E \to Y^{E} \to X^ {E}\to 0$ such that $Y^{E}\in\mathcal{Y}$ and $X^{E}\in\mathcal{X}$.
\end{itemize}
By regarding $\mathcal{E}$ as an extriangulated category, complete hereditary cotorsion pairs in the exact category $\mathcal{E}$ are exactly hereditary cotorsion pairs. 
\end{itemize}
\end{example}

Let $\cotors \mathcal{C}$ denote the set of cotorsion pairs.
We write $(\mathcal{X}_{1},\mathcal{Y}_{1})\le (\mathcal{X}_{2},\mathcal{Y}_{2})$ if $\mathcal{Y}_{1}\subseteq\mathcal{Y}_{2}$.
Then $(\cotors \mathcal{C}, \le)$ clearly becomes a partially ordered set.
We introduce the notions of intervals in $\cotors\mathcal{C}$ and the cohearts of intervals.

\begin{definition}
Let $\mathcal{C}$ be an extriangulated category. 
For $i=1,2$, let $x_{i}:=(\mathcal{X}_{i},\mathcal{Y}_{i})$ $\in \cotors \mathcal{C}$ with $x_{1}\le x_{2}$.
Then we call the subposet
\begin{align}
\cotors [x_{1}, x_{2}]:=\{x\in \cotors \mathcal{C} \mid x_{1} \le x \le x_{2}\} \subseteq \cotors\mathcal{C}\notag
\end{align}
an \emph{interval} in $\cotors \mathcal{C}$ and the subcategory $\mathcal{H}_{[x_{1},x_{2}]} := \mathcal{X}_{1} \cap \mathcal{Y}_{2}\subseteq \mathcal{C}$ the \emph{coheart} of the interval $\cotors [x_{1},x_{2}]$.
Since $\mathcal{H}_{[x_{1},x_{2}]}$ is closed under extensions, $\mathcal{H}_{[x_{1},x_{2}]}$ naturally becomes the extriangulated category (see \cite[Remark 2.18]{NP19}).
\end{definition}

We remark that the coheart of intervals is called the core of twin cotorsion pairs in \cite{LN19}.

In the following, we give a connection between an interval of cotorsion pairs and the poset of cotorsion pairs in the coheart associated to the interval. 
Let $\scotors \mathcal{C}$ denote the poset of $s$-cotorsion pairs and $\scotors[x_{1},x_{2}]:=\scotors\mathcal{C}\cap \cotors[x_{1},x_{2}]$ for $x_{1}\leq x_{2}\in\scotors\mathcal{C}$.
For a subcategory $\mathcal{X}$ of $\mathcal{C}$, let $\add \mathcal{X}$ denote the smallest subcategory of $\mathcal{C}$ containing $\mathcal{X}$ and closed under finite direct sums and direct summands. 
The following theorem is one of main results in this paper. 

\begin{theorem}\label{thm2}
Let $\mathcal{C}$ be an extriangulated category.
For $i=1,2$, let $x_{i}:=(\mathcal{X}_{i},\mathcal{Y}_{i})\in \scotors\mathcal{C}$ with $x_{1}\le x_{2}$.
Then there exist mutually inverse isomorphisms of posets
\[
\begin{tikzcd}
\cotors {[x_{1}, x_{2}]} \rar[shift left, "\Phi"] & \cotors\mathcal{H}_{[x_{1}, x_{2}]}, \lar[shift left, "\Psi"]
\end{tikzcd}
\]
where $\Phi(\mathcal{X},\mathcal{Y}):=(\mathcal{X}\cap\mathcal{Y}_{2}, \mathcal{X}_{1} \cap \mathcal{Y})$ and $\Psi(\mathcal{A},\mathcal{B}):=(\add (\mathcal{X}_{2}\ast\mathcal{A}), \add (\mathcal{B} \ast \mathcal{Y}_{1}))$. 
Moreover, if $\mathbb{E}^{2}(\mathcal{X}_{1}, \mathcal{Y}_{2})=0$, then $\cotors [x_{1}, x_{2}] =\scotors [x_{1}, x_{2}]$.
\end{theorem}

Remark that, by \cite[Lemma 2.4]{LZ}, if $\mathcal{C}$ has enough projective objects, then we can drop the assumption (SCP) for $(\mathcal{X}_{1},\mathcal{Y}_{1})$ in Theorem \ref{thm2} (for detail, see the proof of Proposition \ref{prop_phi}).
Hence Theorem \ref{thm2} recovers \cite[Theorem 4.6]{LZ}.

In the rest of this section, we give a proof of Theorem \ref{thm2}. 
Fix two $s$-cotorsion pairs $x_{1}:=(\mathcal{X}_{1}, \mathcal{Y}_{1}) \le x_{2}:=(\mathcal{X}_{2}, \mathcal{Y}_{2})$ in $\mathcal{C}$. 
For simplicity, let $\mathcal{H}:=\mathcal{H}_{[x_{1},x_{2}]}$.
We show that $\Phi$ and $\Psi$ are well-defined.

\begin{proposition}\label{prop_phi}
If $(\mathcal{X}, \mathcal{Y}) \in \cotors [x_{1}, x_{2}]$, then $(\mathcal{X} \cap \mathcal{Y}_{2}, \mathcal{X}_{1} \cap \mathcal{Y})\in \cotors\mathcal{H}$. 
\end{proposition}

\begin{proof}
Let $(\mathcal{X}, \mathcal{Y}) \in \cotors [x_{1}, x_{2}]$.
Then we show that $(\mathcal{X} \cap \mathcal{Y}_{2}, \mathcal{X}_{1} \cap \mathcal{Y})$ is a cotorsion pair in $\mathcal{H}$.
Clearly we have $\mathcal{X} \cap \mathcal{Y}_{2},  \mathcal{X}_{1} \cap \mathcal{Y} \subseteq \mathcal{H}$. 
Since $\mathcal{X}, \mathcal{Y}_{2}, \mathcal{X}_{1}$ and $\mathcal{Y}$ are closed under direct summands, (CP1) holds. 
By $\mathbb{E}(\mathcal{X}, \mathcal{Y})=0$, we obtain (CP2). 
We only prove (CP3) since the proof of (CP4) is similar. 
Since $\mathcal{H}$ is closed under extensions, it is enough to show $\mathcal{H} \subseteq \cone (\mathcal{X}_{1} \cap \mathcal{Y}, \mathcal{X} \cap \mathcal{Y}_{2}) \cap \mathcal{H}$.
By (CP3) for the cotorsion pair $(\mathcal{X},\mathcal{Y})$, we have $\mathcal{H}\subseteq \cone(\mathcal{Y},\mathcal{X})$.
Since it follows from Lemma \ref{lem_cotors} that $\mathcal{X}_{1}$ is closed under cocones and $\mathcal{Y}_{2}$ is closed under extensions, we obtain $\mathcal{H}\subseteq \cone(\mathcal{X}_{1}\cap \mathcal{Y},\mathcal{X}\cap\mathcal{Y}_{2})$.
This finishes the proof.
\end{proof}

\begin{proposition}\label{prop_psi}
If $(\mathcal{A}, \mathcal{B})\in\cotors\mathcal{H}$, then $(\add (\mathcal{X}_{2} \ast \mathcal{A}), \add( \mathcal{B} \ast \mathcal{Y}_{1})) \in \cotors [x_{1}, x_{2}]$.
\end{proposition}

\begin{proof}
Let $(\mathcal{A}, \mathcal{B}) \in \cotors \mathcal{H}$. 
Then the pair $(\add (\mathcal{X}_{2} \ast \mathcal{A}), \add( \mathcal{B} \ast \mathcal{Y}_{1}))$ clearly satisfies (CP1). 
We show (CP2). 
By $\mathbb{E}(\mathcal{X}_{i}, \mathcal{Y}_{i})=0$ for $i=1,2$, we have  $\mathbb{E}(\mathcal{X}_{2}, \mathcal{B}\ast \mathcal{Y}_{1})=0$ and $\mathbb{E}(\mathcal{A}, \mathcal{B})=0$. 
Thus the assertion follows from Proposition \ref{prop_longex}. 
We prove (CP3). 
This follows from 
\begin{align}
\mathcal{C} 
&=\cone (\mathcal{Y}_{1}, \mathcal{X}_{1})\notag\\
&\subseteq\cone (\mathcal{Y}_{1}, \cone(\mathcal{H}, \mathcal{X}_{2}))&\textnormal{since $\mathcal{X}_{1}$ is closed under cocones}\notag\\
&\subseteq\cone (\mathcal{Y}_{1}, \cone (\cocone (\mathcal{B}, \mathcal{A}),\mathcal{X}_{2})&\textnormal{by $\mathcal{H}\subseteq \cocone(\mathcal{B},\mathcal{A})$}\notag\\
&\subseteq\cone(\mathcal{Y}_{1},\cone(\mathcal{B},\mathcal{X}_{2}\ast \mathcal{A}))&\textnormal{by Lemma \ref{lem_basic}(3)}\notag\\
&\subseteq\cone(\mathcal{B}\ast\mathcal{Y}_{1},\mathcal{X}_{2}\ast\mathcal{A})&\textnormal{by Lemma \ref{lem_basic}(1)}.\notag
\end{align}
Similarly, (CP4) holds. 
\end{proof}

Now we are ready to prove Theorem \ref{thm2}.

\begin{proof}[Proof of Theorem \ref{thm2}]
By Propositions \ref{prop_phi} and \ref{prop_psi}, the maps $\Phi$ and $\Psi$ are well-defined.
Moreover, it is clear that these maps are order-preserving.
We show that $\Phi$ and $\Psi$ are mutually inverse isomorphisms.
Let $(\mathcal{X}, \mathcal{Y}) \in \cotors[x_{1}, x_{2}]$. 
Then 
\begin{align}
\add (\mathcal{X}_{2} \ast (\mathcal{X} \cap \mathcal{Y}_{2})) \subseteq \mathcal{X},\  \add ((\mathcal{X}_{1} \cap \mathcal{Y}) \ast \mathcal{Y}_{1})\subseteq \mathcal{Y}. \notag
\end{align}
Since $(\mathcal{X}, \mathcal{Y})$ and $\Psi\Phi (\mathcal{X}, \mathcal{Y})$ are cotorsion pairs, $\Psi\Phi (\mathcal{X}, \mathcal{Y})=(\mathcal{X}, \mathcal{Y})$ holds by Lemma \ref{lem_cotors}(1).
Let $(\mathcal{A}, \mathcal{B})$ be a cotorsion pair in $\mathcal{H}$.
We put $(\mathcal{A}', \mathcal{B}') :=\Phi \Psi (\mathcal{A}, \mathcal{B})$.  
It is enough to show that $\mathcal{B}=\mathcal{B}'$. 
Let $B \in \mathcal{B}$. 
By $\mathcal{H}\subseteq\cocone (\mathcal{B}', \mathcal{A}')$, we have an $\mathfrak{s}$-conflation $B \to B' \to A' \dashrightarrow$ with $B' \in \mathcal{B}'$ and $A' \in \mathcal{A}'$. 
Since $\mathbb{E}(\mathcal{X}_{2}, \mathcal{B})=0$ and $\mathbb{E}(\mathcal{A}, \mathcal{B})=0$, the $\mathfrak{s}$-conflation splits. 
Thus we have $B \in \mathcal{B}'$. 
Similarly, we have the converse inclusion. 
Thus the former assertion holds. 
We assume $\mathbb{E}^{2}(\mathcal{X}_{1}, \mathcal{Y}_{2})=0$.
Then, for each $(\mathcal{X}, \mathcal{Y}) \in \cotors
\mathcal[x_{1}, x_{2}]$, we have $\mathbb{E}^{2}(\mathcal{X}, \mathcal{Y})=0$, and hence $(\mathcal{X},\mathcal{Y})\in \scotors\mathcal{C}$.
This finishes the proof.
\end{proof}
Let $\mathcal{D}$ be a triangulated category.
For co-$t$-structures $(\mathcal{U}_{1}, \mathcal{V}_{1}), (\mathcal{U}_{2}, \mathcal{V}_{2})$ on $\mathcal{D}$ with $\mathcal{V}_{1} \subseteq \mathcal{V}_{2}$, let
\begin{align}
\cotstr [(\mathcal{U}_{1}, \mathcal{V}_{1}), (\mathcal{U}_{2}, \mathcal{V}_{2})]:= \{\textnormal{$(\mathcal{U}, \mathcal{V})$: co-$t$-structure on $\mathcal{D}$} \mid \mathcal{V}_{1} \subseteq \mathcal{V} \subseteq \mathcal{V}_{2}\}. \notag
\end{align}
By Theorem \ref{thm2}, we have the following result, which recovers \cite[Theorem 2.1]{PZ}.

\begin{corollary}\label{cor_PZ}
Let $\mathcal{D}$ be a triangulated category with shift functor $\Sigma$.
For $i=1,2$, let $(\mathcal{U}_{i},\mathcal{V}_{i})$ be a co-$t$-structure on $\mathcal{D}$ with $\mathcal{V}_{1} \subseteq \mathcal{V}_{2}$ and $\mathcal{H}:=\mathcal{U}_{1} \cap \mathcal{V}_{2}$. 
Assume that $\mathcal{D}(\mathcal{U}_{1}, \Sigma^{2}\mathcal{V}_{2})=0$. 
Then there exist mutually inverse isomorphisms of posets
\[
\begin{tikzcd}
\cotstr {[}(\mathcal{U}_{1},\mathcal{V}_{1}),(\mathcal{U}_{2},\mathcal{V}_{2}){]} \rar[shift left, "\Phi"] & \cotors\mathcal{H}, \lar[shift left, "\Psi"]
\end{tikzcd}
\]
where $\Phi(\mathcal{U},\mathcal{V}):=(\mathcal{U}\cap\mathcal{V}_{2}, \mathcal{U}_{1} \cap\mathcal{V})$ and $\Psi(\mathcal{A},\mathcal{B}):=(\add(\mathcal{U}_{2}\ast\mathcal{A}), \add(\mathcal{B} \ast \mathcal{V}_{1}))$.
\end{corollary}

\begin{proof}
We regard $\mathcal{D}$ as an extriangulated category. 
For $i=1,2$, the pair $x_{i}:=(\mathcal{U}_{i}, \mathcal{V}_{i})$ is an $s$-cotorsion pair by Lemma \ref{lem_shift} and Example \ref{ex_scotors}(1). 
Since $\mathbb{E}^{2}(\mathcal{U}_{1}, \mathcal{V}_{2})=\mathcal{D}(\mathcal{U}_{1}, \Sigma^{2}\mathcal{V}_{2})=0$, the assertion follows from  Theorem \ref{thm2}.
\end{proof}

\section{Properties of subcategories $\mathcal{X}^{\htt}$ and $\mathcal{X}^{\chk}$}

Let $\mathcal{C}=(\mathcal{C},\mathbb{E},\mathfrak{s})$ be an extriangulated category. 
In this section, we study properties of $\mathcal{X}^{\htt}$ and $\mathcal{X}^{\chk}$ for a subcategory $\mathcal{X}$ of $\mathcal{C}$. 
We start this section with giving the definition of subcategories $\mathcal{X}^{\htt}$ and $\mathcal{X}^{\chk}$. 

\begin{definition}
Let $\mathcal{X}$ be a subcategory of $\mathcal{C}$. 
For each $n\geq 0$, we inductively define subcategories $\mathcal{X}^{\htt}_{n}$ and $\mathcal{X}^{\chk}_{n}$ of $\mathcal{C}$ as $\mathcal{X}^{\htt}_{n}:=\cone(\mathcal{X}^{\htt}_{n-1},\mathcal{X})$ and $\mathcal{X}^{\chk}_{n}:=\cocone(\mathcal{X},\mathcal{X}^{\chk}_{n-1})$, where $\mathcal{X}^{\htt}_{-1}:=\{ 0 \}$ and $\mathcal{X}^{\chk}_{-1}:=\{ 0 \}$.
Put 
\begin{align}
\displaystyle \mathcal{X}^{\htt}:=\bigcup_{n\geq 0}\mathcal{X}^{\htt}_{n},\ \ \mathcal{X}^{\chk}:=\bigcup_{n\geq 0}\mathcal{X}^{\chk}_{n}.\notag
\end{align}
We define a subcategory $\mathcal{X}^{\tld}$ of $\mathcal{C}$ as $\mathcal{X}^{\tld}:=(\mathcal{X}^{\htt})^{\chk}$.
\end{definition}

When $\mathcal{C}$ is a triangulated category, descriptions of $\mathcal{X}^{\htt}$ and  $\mathcal{X}^{\chk}$ are well-known.

\begin{remark}\label{rem_filt-hat}
Let $\mathcal{D}$ be a triangulated category (regarded as an extriangulated category) with shift functor $\Sigma$. For a subcategory $\mathcal{X}$ and an integer $n\geq 0$, we obtain 
\begin{align}
&\mathcal{X}^{\htt}_{n}=\mathcal{X}\ast \Sigma\mathcal{X}\ast\cdots\ast\Sigma^{n}\mathcal{X},\notag\\
&\mathcal{X}^{\chk}_{n}=\Sigma^{-n}\mathcal{X}\ast \Sigma^{-n+1}\mathcal{X}\ast \cdots \ast \mathcal{X}.\notag
\end{align}
If $\mathcal{X}$ is closed under extensions and negative shifts, then  $\mathcal{X}^{\htt}_{n}=\Sigma^{n}\mathcal{X}$ holds. 
Similarly, if $\mathcal{X}$ is closed under extensions and positive shifts, then  $\mathcal{X}^{\chk}_{n}=\Sigma^{-n}\mathcal{X}$ holds. 
\end{remark}

We give an easy observation of $\mathcal{X}^{\htt}$ and $\mathcal{X}^{\chk}$. 

\begin{lemma}\label{lem_conecl}
Let $\mathcal{X}, \mathcal{Y}$ be subcategories of $\mathcal{C}$ satisfying $\mathcal{X}\subseteq\mathcal{Y}$. 
If $\mathcal{Y}$ is closed under cones, then  $\mathcal{X}^{\htt}\subseteq \mathcal{Y}$.
If $\mathcal{Y}$ is closed under cocones,  then  $\mathcal{X}^{\chk}\subseteq\mathcal{Y}$.
\end{lemma}

\begin{proof}
We show $\mathcal{X}^{\htt}_{n} \subseteq\mathcal{Y}$ by induction on $n$. 
If $n=0$, then this is clear. 
Let $n\ge1$. 
Then we have $\mathcal{X}^{\htt}_{n}=\cone(\mathcal{X}^{\htt}_{n-1}, \mathcal{X})\subseteq\cone (\mathcal{Y},\mathcal{Y})=\mathcal{Y}$, where the middle inclusion follows from the induction hypothesis and the last equality follows from the assumption that $\mathcal{Y}$ is closed under cones. 
Hence $\mathcal{X}^{\htt}\subseteq\mathcal{Y}$ holds.
Similarly, we obtain $\mathcal{X}^{\chk}\subseteq\mathcal{Y}$. \end{proof}

The following lemma is frequently used in the rest of this paper. 

\begin{lemma}\label{lem_perp}
Let $\mathcal{X}$ be a subcategory of $\mathcal{C}$ and let $n$ be a non-negative integer. 
Then we have ${}^{\perp}\mathcal{X}={}^{\perp}(\mathcal{X}^{\htt}_{n})$ and $\mathcal{X}^{\perp}=(\mathcal{X}^{\chk}_{n})^{\perp}$. 
\end{lemma}

\begin{proof}
Since $\mathcal{X}\subseteq\mathcal{X}^{\htt}_{n}$ clearly holds, we obtain ${}^{\perp}\mathcal{X}\supseteq {}^{\perp}(\mathcal{X}^{\htt}_{n})$.
By induction on $n$, we show $\mathbb{E}^{k}({}^{\perp}\mathcal{X}, \mathcal{X}^{\htt}_{n})=0$ for each $k \ge 1$.
If $n=0$, then this is clear.
Assume $n \ge 1$. Let $M \in \mathcal{X}^{\htt}_{n}$. 
Then we have an $\mathfrak{s}$-conflation $K \to X \to M \dashrightarrow$ with $K \in \mathcal{X}^{\htt}_{n-1}$ and $X \in \mathcal{X}$.
Applying $\mathcal{C}({}^{\perp} \mathcal{X}, -)$ to the $\mathfrak{s}$-conflation gives an exact sequence
\begin{align}
\mathbb{E}^{k}({}^{\perp} \mathcal{X}, X) \to \mathbb{E}^{k}({}^{\perp} \mathcal{X}, M) \to \mathbb{E}^{k+1}({}^{\perp} \mathcal{X}, K) \notag
\end{align}
for all $k\geq 1$.
By $X \in \mathcal{X}$ and the induction hypothesis, the left-hand side and right-hand side vanish respectively.
Hence the assertion holds.
Similarly, we have $\mathcal{X}^{\perp}=(\mathcal{X}^{\chk}_{n})^{\perp}$. 
\end{proof}

The following lemma gives a sufficient condition of $\mathcal{X}^{\htt}$ to be closed under direct summands. 

\begin{lemma}\label{lem_cldir}
Let $\mathcal{X}$ be a subcategory of $\mathcal{C}$ and let $n$ be a non-negative integer. 
Assume that $\mathcal{X}$ is closed under direct summands and $\mathcal{X}^{\htt} \subseteq \cone (\mathcal{X}^{\perp}, \mathcal{X})$.  
Then the following statements hold. 
\begin{itemize}
\item[(1)] $\mathcal{X}^{\htt}_{n}=\{M \in \mathcal{X}^{\htt} \mid \mathbb{E}^{k}(M, \mathcal{X}^{\perp})=0\textnormal{\; for\;each\;}k\ge n+1\}$. 
\item[(2)] If $\mathcal{X}^{\htt}_{k}$ is closed under extensions for each $k \le n$, then 
$\mathcal{X}^{\htt}_{n}$ is closed under direct summands.
\end{itemize}
\end{lemma}

\begin{proof}
(1) Let $M \in \mathcal{X}^{\htt}_{n}$. 
By induction on $n$, we show $\mathbb{E}^{k}(M, \mathcal{X}^{\perp})=0$ for each $k\ge n+1$. 
If $n=0$, then the assertion clearly holds. 
Let $n \ge 1$. 
Then we have an $\mathfrak{s}$-conflation $K \to X \to M \dashrightarrow$ with $K \in \mathcal{X}^{\htt}_{n-1}$ and $X \in \mathcal{X}$. 
Applying $\mathcal{C}(-,\mathcal{X}^{\perp})$ to the $\mathfrak{s}$-conflation, we have an isomorphism $\mathbb{E}^{k}(K,\mathcal{X}^{\perp})\cong\mathbb{E}^{k+1}(M,\mathcal{X}^{\perp})$ for each $k \ge 1$.
By the induction hypothesis, we obtain $\mathbb{E}^{k}(K, \mathcal{X}^{\perp})=0$ for each $k \ge n$. 
Thus the assertion holds. 
Conversely, let $M \in \mathcal{X}^{\htt}$ with $\mathbb{E}^{k}(M, \mathcal{X}^{\perp})=0$ for each $k \ge n+1$. 
By induction on $n$, we show $M \in \mathcal{X}^{\htt}_{n}$. 
Let $n=0$. 
Since $\mathcal{X}^{\htt} \subseteq \cone (\mathcal{X}^{\perp}, \mathcal{X})$, we have an $\mathfrak{s}$-conflation $A \to B \to M \dashrightarrow$ with $A \in \mathcal{X}^{\perp}$ and $B \in \mathcal{X}$. 
By $\mathbb{E}(M, \mathcal{X}^{\perp})=0$, the $\mathfrak{s}$-conflation splits, and hence $M \in \mathcal{X}$.
Let $n \ge 1$. 
Since $M \in \mathcal{X}^{\htt}$, we have an $\mathfrak{s}$-conflation $K \to X \to M \dashrightarrow$ with $K \in \mathcal{X}^{\htt}$ and $X \in \mathcal{X}$.
Applying $\mathcal{C}(-,\mathcal{X}^{\perp})$ to the $\mathfrak{s}$-conflation gives an isomorphism $\mathbb{E}^{k}(K, \mathcal{X}^{\perp})\cong \mathbb{E}^{k+1}(M,\mathcal{X}^{\perp})=0$ for each $k \ge n$. 
By the induction hypothesis, $K \in \mathcal{X}^{\htt}_{n-1}$. 
Thus the proof is complete.

(2) Let $M:=M_{1} \oplus M_{2} \in \mathcal{X}^{\htt}_{n}$.
By induction on $n$, we show $M_{1}, M_{2} \in \mathcal{X}^{\htt}_{n}$. 
If $n=0$, then the assertion clearly holds. 
Let $n\ge 1$.
Then we have an $\mathfrak{s}$-conflation $K \to X \to M \dashrightarrow$ with $K \in \mathcal{X}^{\htt}_{n-1}$ and $X\in \mathcal{X}$. 
Applying (ET4)$^{\op}$ to the $\mathfrak{s}$-conflation and $M_{2} \xrightarrow{\tiny\left[\begin{smallmatrix}0 \\1  \end{smallmatrix}\right]}M_{1} \oplus M_{2} \xrightarrow{\tiny\left[\begin{smallmatrix}1&0 \end{smallmatrix}\right]}M_{1} \dashrightarrow$, we obtain $\mathfrak{s}$-conflations $K_{1}\xrightarrow{a_{1}}X\xrightarrow{b_{1}} M_{1}\dashrightarrow$ and $K\xrightarrow{c_{1}}K_{1}\xrightarrow{d_{1}}M_{2}\dashrightarrow$.
Similarly, we have $\mathfrak{s}$-conflations $K_{2} \xrightarrow{a_{2}} X \xrightarrow{b_{2}} M_{2} \dashrightarrow$ and $K \xrightarrow{c_{2}} K_{2} \xrightarrow{d_{2}} M_{1} \dashrightarrow$. 
Since $\mathcal{X}^{\htt}_{n}$ is closed under extensions, 
the $\mathfrak{s}$-conflation $K \oplus K \xrightarrow{\tiny\left[\begin{smallmatrix}c_{1} &0 \\0 &c_{2}  \end{smallmatrix}\right]} K_{1} \oplus K_{2} \xrightarrow{\tiny\left[\begin{smallmatrix}d_{1} & 0\\0&d_{2}  \end{smallmatrix}\right]}M_{2} \oplus M_{1} \dashrightarrow$ induces $K_{1} \oplus K_{2} \in \mathcal{X}^{\htt}_{n}$. 
Moreover, applying $\mathcal{C}(-, \mathcal{X}^{\perp})$ to the $\mathfrak{s}$-conflation $K_{1} \oplus K_{2} \xrightarrow{\tiny\left[\begin{smallmatrix}a_{1} &0 \\0 &a_{2}  \end{smallmatrix}\right]} X \oplus X \xrightarrow{\tiny\left[\begin{smallmatrix}b_{1} & 0\\0&b_{2}  \end{smallmatrix}\right]}M_{1} \oplus M_{2} \dashrightarrow$ gives an isomorphism
\begin{align}
\mathbb{E}^{k}(K_{1} \oplus K_{2}, \mathcal{X}^{\perp}) \cong \mathbb{E}^{k+1}(M_{1} \oplus M_{2}, \mathcal{X}^{\perp})\notag
\end{align}
for each $k \ge 1$. 
If $k\ge n$, then the right-hand side vanishes by (1).
Therefore it follows from (1) that $K_{1} \oplus K_{2} \in \mathcal{X}^{\htt}_{n-1}$. 
Hence the induction hypothesis gives $K_{1}, K_{2} \in \mathcal{X}^{\htt}_{n-1}$. 
By the $\mathfrak{s}$-conflation $K_{i} \to X \to M_{i}\dashrightarrow$ satisfying $K_{i}\in \mathcal{X}^{\htt}_{n-1}$ and $X\in\mathcal{X}$, we have $M_{i}\in \mathcal{X}^{\htt}_{n}$. 
This finishes the proof.
\end{proof}

Following \cite[Section 3]{AB89} (see also \cite{MDZ} and \cite{MLHG}), we collect basic properties of a subcategory $\mathcal{X}$ with a cogenerator $\mathcal{W}$, that is, $\mathcal{X}$ and $\mathcal{W}$ are subcategories of $\mathcal{C}$ satisfying $\mathcal{W}\subseteq \mathcal{X}\subseteq\cocone(\mathcal{W},\mathcal{X})$.
For the convenience of the readers, we give the proof of the following results.  
Under certain conditions, we give descriptions of $\mathcal{X}_{n}^{\htt}$ and $\mathcal{W}_{n}^{\htt}$.

\begin{lemma}\label{lem_htt-cone}
Let $\mathcal{X},\mathcal{W}$ be subcategories of $\mathcal{C}$ and let $n$ be a non-negative integer. Assume $\mathcal{W}\subseteq \mathcal{X}\subseteq\cocone(\mathcal{W},\mathcal{X})$.
Then the following statements hold.
\begin{itemize}
\item[(1)] If $\mathcal{X}$ is closed under extensions, then we obtain the following statements.
\begin{itemize}
\item[(a)] $\mathcal{X}^{\htt}_{n}=\cone (\mathcal{W}^{\htt}_{n-1}, \mathcal{X}) \subseteq \cocone (\mathcal{W}^{\htt}_{n}, \mathcal{X})$.
\item[(b)] $\mathcal{X}$ is closed under cocones if and only if $\mathcal{X}=\cocone(\mathcal{W},\mathcal{X})$.
In this case, we have $\mathcal{X}^{\htt}_{n}=\cocone(\mathcal{W}^{\htt}_{n},\mathcal{X})$.
\end{itemize}
\item[(2)] If $\mathcal{W}$ is closed under direct summands and $\mathcal{W}\subseteq\mathcal{X}^{\perp}$, then
\begin{align}
\mathcal{W}^{\htt}_{n}=\cone(\mathcal{W}^{\htt}_{n-1},\mathcal{X})\cap \mathcal{W}^{\htt}=\cone(\mathcal{W}^{\htt}_{n-1},\mathcal{X})\cap \mathcal{X}^{\perp}.\notag
\end{align}
\end{itemize}
\end{lemma}

\begin{proof}
(1) We prove (a) by induction on $n$. 
If $n=0$, then the assertion clearly holds. 
We assume $n \ge 1$. 
Then 
\begin{align}
\mathcal{X}^{\htt}_{n}
&=\cone(\mathcal{X}^{\htt}_{n-1},\mathcal{X})&\textnormal{by definition}\notag\\
&\subseteq \cone(\cocone (\mathcal{W}^{\htt}_{n-1}, \mathcal{X}),\mathcal{X})&\textnormal{by the induction hypothesis} \notag\\
&\subseteq \cone(\mathcal{W}^{\htt}_{n-1},  \mathcal{X} \ast \mathcal{X})&\textnormal{by Lemma \ref{lem_basic}(3)}\notag\\
&= \cone(\mathcal{W}^{\htt}_{n-1},  \mathcal{X})&\textnormal{since $\mathcal{X}$ is closed under extensions}\notag\\
&\subseteq \cone(\mathcal{W}^{\htt}_{n-1},  \cocone (\mathcal{W}, \mathcal{X}))&\textnormal{by assumption}\notag\\
&= \cocone(\cone (\mathcal{W}^{\htt}_{n-1}, \mathcal{W}), \mathcal{X})&\textnormal{by Lemma \ref{lem_basic}(7)}\notag\\
&=\cocone (\mathcal{W}^{\htt}_{n},\mathcal{X})&\textnormal{by definition.}\notag
\end{align}
Moreover, by $\cone (\mathcal{W}^{\htt}_{n-1}, \mathcal{X})\subseteq \cone(\mathcal{X}^{\htt}_{n-1},\mathcal{X})=\mathcal{X}^{\htt}_{n}$, we have the assertion.

We prove (b). If $\mathcal{X}$ is closed under cocones, then we clearly have $\mathcal{X}=\cocone(\mathcal{W}, \mathcal{X})$. 
Conversely, let $\mathcal{X}=\cocone (\mathcal{W}, \mathcal{X})$. Then 
\begin{align}
\cocone (\mathcal{X}, \mathcal{X})
&= \cocone(\cocone (\mathcal{W}, \mathcal{X}), \mathcal{X})&\textnormal{by assumption}\notag\\
&\subseteq \cocone(\mathcal{W}, \mathcal{X} \ast \mathcal{X})&\textnormal{by Lemma \ref{lem_basic}(2)}\notag\\
&= \cocone(\mathcal{W}, \mathcal{X} )&\textnormal{since $\mathcal{X}$ is closed under extensions}\notag\\
&= \mathcal{X}&\textnormal{by assumption}.\notag
\end{align}
In this case, we have
\begin{align}
\mathcal{X}^{\htt}_{n}
&=\cone (\mathcal{W}^{\htt}_{n-1}, \mathcal{X})&\textnormal{by (a)}\notag\\
&=\cone(\mathcal{W}^{\htt}_{n-1}, \cocone(\mathcal{W}, \mathcal{X}))&\textnormal{by assumption}\notag\\
&=\cocone(\cone(\mathcal{W}^{\htt}_{n-1},\mathcal{W}),\mathcal{X})&\textnormal{by Lemma \ref{lem_basic}(7)}\notag\\
&=\cocone (\mathcal{W}^{\htt}_{n}, \mathcal{X}) &\textnormal{by definition}.\notag 
\end{align}

(2) By Lemma \ref{lem_perp}, $\mathcal{W}\subseteq \mathcal{X}^{\perp}$ implies $\mathcal{W}^{\htt}\subseteq \mathcal{X}^{\perp}$.
Thus we have
\begin{align}
\mathcal{W}^{\htt}_{n}\subseteq\cone(\mathcal{W}^{\htt}_{n-1},\mathcal{X})\cap\mathcal{W}^{\htt}\subseteq \cone(\mathcal{W}^{\htt}_{n-1},\mathcal{X})\cap\mathcal{X}^{\perp}.\notag
\end{align}
It is enough to show $\cone(\mathcal{W}^{\htt}_{n-1},\mathcal{X})\cap\mathcal{X}^{\perp}\subseteq\mathcal{W}^{\htt}_{n}$.
Let $M\in \cone(\mathcal{W}^{\htt}_{n-1},\mathcal{X})\cap\mathcal{X}^{\perp}$.
Then there exists an $\mathfrak{s}$-conflation $N\to X\to M\dashrightarrow$ with $N\in\mathcal{W}^{\htt}_{n-1}\subseteq \mathcal{X}^{\perp}$ and $X\in\mathcal{X}$.
Since $\mathcal{X}^{\perp}$ is closed under extensions, we have $X\in\mathcal{X}^{\perp}$. 
On the other hand, by $\mathcal{X}\subseteq\cocone(\mathcal{W},\mathcal{X})$, there exists an $\mathfrak{s}$-conflation $X\to W\to X'\dashrightarrow$ with $W\in\mathcal{W}$ and $X'\in\mathcal{X}$. 
Thus it follows from $X\in \mathcal{X}^{\perp}$ that the $\mathfrak{s}$-conflation splits.
Since $\mathcal{W}$ is closed under direct summands, we obtain $X\in\mathcal{W}$.
The proof is complete.
\end{proof}

We collect basic properties of $\mathcal{X}^{\htt}$. 

\begin{proposition}\label{prop_htt}
Let $\mathcal{X},\mathcal{W}$ be subcategories of $\mathcal{C}$ and let $n$ be a non-negative integer. 
Assume that $\mathcal{X}$ is closed under extensions, $\mathcal{W} \subseteq \mathcal{X} \cap \mathcal{X}^{\perp}$ and $\mathcal{X}\subseteq \cocone (\mathcal{W}, \mathcal{X})$. 
Then the following statements hold. 
\begin{itemize}
\item[(1)] $\mathcal{X}^{\htt}_{n}$ is closed under extensions. 
\item[(2)] $\cone (\mathcal{X}^{\htt}_{n}, \mathcal{X}^{\htt}_{n}) =\mathcal{X}^{\htt}_{n+1}$. 
\item[(3)] If $\mathcal{X}$ is closed under direct summands, then so is $\mathcal{X}^{\htt}_{n}$. 
Moreover, $\mathcal{X}^{\htt}$ is the smallest subcategory of $\mathcal{C}$ containing $\mathcal{X}$ and closed under extensions, cones and direct summands.
\item[(4)] If $\mathcal{X}$ is closed under cocones, then $\cocone (\mathcal{X}^{\htt}_{n}, \mathcal{X}^{\htt}_{n}) = \mathcal{X}^{\htt}_{n}$. 
\item[(5)] If $\mathcal{X}$ is closed under cocones and direct summands, then $\mathcal{X}^{\htt}$ is the smallest subcategory of $\mathcal{C}$ containing $\mathcal{X}$ and closed under extensions, cones, cocones and direct summands, that is, $\mathcal{X}^{\htt}=\thick \mathcal{X}$.
\end{itemize}
\end{proposition}

\begin{proof}
(1) First we prove
\begin{align}\label{right-cl}
\mathcal{X}^{\htt}_{n} \ast \mathcal{X} \subseteq \mathcal{X}^{\htt}_{n}. 
\end{align}
Indeed, we have 
\begin{align}
\mathcal{X}^{\htt}_{n}\ast \mathcal{X}
&= \cone(\mathcal{W}^{\htt}_{n-1}, \mathcal{X}) \ast \mathcal{X}&\textnormal{by Lemma \ref{lem_htt-cone}(1-a)} \notag\\
&\subseteq \cone(\mathcal{W}^{\htt}_{n-1}, \mathcal{X}\ast \mathcal{X})&\textnormal{by Lemmas \ref{lem_basic}(8) and \ref{lem_perp}}\notag\\
&= \cone(\mathcal{W}^{\htt}_{n-1}, \mathcal{X}) &\textnormal{since $\mathcal{X}$ is closed under extensions}\notag\\
&=\mathcal{X}^{\htt}_{n} &\textnormal{by Lemma \ref{lem_htt-cone}(1-a)}\notag.
\end{align}
We show that $\mathcal{X}^{\htt}_{n}$ is closed under extensions by induction on $n$. 
If $n=0$, then this is clear. 
Let $n \ge 1$. 
Then 
\begin{align}
\mathcal{X}^{\htt}_{n} \ast \mathcal{X}^{\htt}_{n}
&=\mathcal{X}^{\htt}_{n}\ast \cone(\mathcal{W}^{\htt}_{n-1},\mathcal{X})&\textnormal{by Lemma \ref{lem_htt-cone}(1-a)}\notag\\
&\subseteq \cone(\mathcal{W}^{\htt}_{n-1},\mathcal{X}^{\htt}_{n}\ast \mathcal{X})&\textnormal{by Lemma \ref{lem_basic}(5)}\notag\\
&\subseteq \cone(\mathcal{W}^{\htt}_{n-1},\mathcal{X}^{\htt}_{n})&\textnormal{by \eqref{right-cl}}\notag\\
&= \cone(\mathcal{W}^{\htt}_{n-1},\cone (\mathcal{W}^{\htt}_{n-1}, \mathcal{X}))&\textnormal{by Lemma \ref{lem_htt-cone}(1-a)}\notag\\
&= \cone(\mathcal{W}^{\htt}_{n-1}\ast \mathcal{W}^{\htt}_{n-1}, \mathcal{X})&\textnormal{by Lemma \ref{lem_basic}(1)}\notag\\
&\subseteq\cone(\mathcal{X}^{\htt}_{n-1}\ast \mathcal{X}^{\htt}_{n-1}, \mathcal{X})&\textnormal{by $\mathcal{W}^{\htt}_{n-1}\subseteq\mathcal{X}^{\htt}_{n-1}$}\notag\\
&= \cone(\mathcal{X}^{\htt}_{n-1}, \mathcal{X})=\mathcal{X}^{\htt}_{n}&\textnormal{by the induction hypothesis}. \notag
\end{align}
Thus we have the assertion. 

(2) Since
\begin{align}
\cone(\mathcal{X}^{\htt}_{n},\mathcal{X})
&\subseteq \cone (\mathcal{X}^{\htt}_{n}, \mathcal{X}^{\htt}_{n})&\textnormal{by $\mathcal{X}\subseteq\mathcal{X}^{\htt}_{n}$}\notag\\
&= \cone(\mathcal{X}^{\htt}_{n}, \cone(\mathcal{X}^{\htt}_{n-1}, \mathcal{X}))&\textnormal{by definition}\notag\\
&\subseteq \cone(\mathcal{X}^{\htt}_{n-1}\ast\mathcal{X}^{\htt}_{n},\mathcal{X})&\textnormal{by Lemma \ref{lem_basic}(1)}\notag\\
&\subseteq \cone(\mathcal{X}^{\htt}_{n},  \mathcal{X})&\textnormal{by (1)}, \notag
\end{align}
we have $\mathcal{X}^{\htt}_{n+1}=\cone(\mathcal{X}^{\htt}_{n},\mathcal{X})=\cone(\mathcal{X}^{\htt}_{n},\mathcal{X}^{\htt}_{n})$.

(3) Assume that $\mathcal{X}$ is closed under direct summands.
By Lemmas \ref{lem_perp} and \ref{lem_htt-cone}(1-a), we have
\begin{align}
\mathcal{X}^{\htt}_{n} =\cone (\mathcal{W}^{\htt}_{n-1}, \mathcal{X})\subseteq \cone (\mathcal{X}^{\perp}, \mathcal{X}).  \notag 
\end{align}
Moreover, it follows from (1) that  $\mathcal{X}^{\htt}_{k}$ is closed under extensions for each $k\leq n$.
By Lemma \ref{lem_cldir}(2), $\mathcal{X}^{\htt}_{n}$ is closed under direct summands.
Hence $\mathcal{X}^{\htt}$ is the smallest subcategory of $\mathcal{C}$ containing $\mathcal{X}$ and closed under extensions, cones and direct summands by (1), (2) and Lemma \ref{lem_conecl}.

(4) Assume that $\mathcal{X}$ is closed under cocones. 
Then we have
\begin{align}
\cocone (\mathcal{X}^{\htt}_{n}, \mathcal{X}^{\htt}_{n})
&= \cocone(\mathcal{X}^{\htt}_{n}, \cone(\mathcal{X}^{\htt}_{n-1}, \mathcal{X}))&\textnormal{by definition}\notag\\
&\subseteq \cocone(\mathcal{X}^{\htt}_{n-1}\ast\mathcal{X}^{\htt}_{n},\mathcal{X})&\textnormal{by Lemma \ref{lem_basic}(4)}\notag\\
&\subseteq \cocone(\mathcal{X}^{\htt}_{n},  \mathcal{X})&\textnormal{by (1)} \notag\\
&= \cocone(\cocone (\mathcal{W}^{\htt}_{n}, \mathcal{X}),  \mathcal{X})&\textnormal{by Lemma \ref{lem_htt-cone}(1-b)}\notag\\
&\subseteq \cocone (\mathcal{W}^{\htt}_{n}, \mathcal{X}\ast  \mathcal{X})&\textnormal{by Lemma \ref{lem_basic}(2)}\notag\\
&= \cocone (\mathcal{W}^{\htt}_{n}, \mathcal{X})&\textnormal{since $\mathcal{X}$ is closed under extensions}\notag\\
&=\mathcal{X}^{\htt}_{n}&\textnormal{by Lemma \ref{lem_htt-cone}(1-b)}. \notag
\end{align}
Hence the assertion holds.

(5) By (3) and (4), $\mathcal{X}^{\htt}$ is a thick subcategory of $\mathcal{C}$ containing $\mathcal{X}$. 
Hence $\mathcal{X}^{\htt} \supseteq\thick\mathcal{X}$. 
The converse inclusion follows from Lemma \ref{lem_conecl}.
\end{proof}

Dually we have the following result.

\begin{proposition}\label{prop_chk}
Let $\mathcal{X},\mathcal{W}$ be subcategories of $\mathcal{C}$ and let $m$ be a non-negative integer. 
Assume that $\mathcal{X}$ is closed under extensions, $\mathcal{W} \subseteq \mathcal{X} \cap {}^{\perp}\mathcal{X}$ and $\mathcal{X}\subseteq \cone (\mathcal{X}, \mathcal{W})$. 
Then the following statements hold. 
\begin{itemize}
\item[(1)] $\mathcal{X}^{\chk}_{m}$ is closed under extensions. 
\item[(2)] $\cocone (\mathcal{X}^{\chk}_{m}, \mathcal{X}^{\chk}_{m}) =\mathcal{X}^{\chk}_{m+1}$. 
\item[(3)] If $\mathcal{X}$ is closed under direct summands, then so is $\mathcal{X}^{\chk}_{m}$. 
Moreover, $\mathcal{X}^{\chk}$ is the smallest subcategory of $\mathcal{C}$ containing $\mathcal{X}$ and closed under extensions, cocones and direct summands.
\item[(4)] If $\mathcal{X}$ is closed under cones, then $\cone (\mathcal{X}^{\chk}_{m}, \mathcal{X}^{\chk}_{m}) = \mathcal{X}^{\chk}_{m}$. 
\item[(5)] If $\mathcal{X}$ is closed under cones and direct summands, then $\mathcal{X}^{\chk}$ is the smallest subcategory of $\mathcal{C}$ containing $\mathcal{X}$ and closed under extensions, cones, cocones and direct summands, that is, $\mathcal{X}^{\chk}=\thick \mathcal{X}$.
\end{itemize}
\end{proposition}

In the following, we study properties of self-orthogonal subcategories, which play an important role.
A subcategory $\mathcal{M}$ of $\mathcal{C}$ is said to be \emph{self-orthogonal} if $\mathbb{E}^{k}(\mathcal{M},\mathcal{M})=0$ holds for all $k\ge 1$. 
A self-orthogonal subcategory $\mathcal{M}$ is called a \emph{presilting subcategory} if it is closed under direct summands.
We give easy observations for self-orthogonal subcategories. 

\begin{lemma}\label{lem_selfort}
Let $\mathcal{M}$ be a self-orthogonal subcategory of $\mathcal{C}$ and let $n,m$ be non-negative integers.
Then the following statements hold.
\begin{itemize}
\item[(1)] $\mathcal{M}$ is closed under extensions.
\item[(2)] $\mathcal{M}\subseteq\mathcal{M}^{\htt}_{n}\cap{}^{\perp}(\mathcal{M}^{\htt}_{n})$ and $\mathcal{M}^{\htt}_{n}\subseteq\cone(\mathcal{M}^{\htt}_{n},\mathcal{M})$.
\item[(3)] $\mathcal{M}\subseteq\mathcal{M}^{\chk}_{m}\cap(\mathcal{M}^{\chk}_{m})^{\perp}$ and $\mathcal{M}^{\chk}_{m}\subseteq\cocone(\mathcal{M},\mathcal{M}^{\chk}_{m})$.
\item[(4)] $\mathbb{E}^{k}(\mathcal{M}^{\chk}_{m}, \mathcal{M}^{\htt}_{n})=0$ for each $k \ge 1$. 
\item[(5)] $(\mathcal{M}^{\htt}_{n})^{\chk}_{m}=\cocone (\mathcal{M}^{\htt}_{n}, \mathcal{M}^{\chk}_{m-1})
=\cone (\mathcal{M}^{\htt}_{n-1}, \mathcal{M}^{\chk}_{m})
=(\mathcal{M}^{\chk}_{m})^{\htt}_{n}$. In particular, $\mathcal{M}^{\tld}=(\mathcal{M}^{\htt})^{\chk}=(\mathcal{M}^{\chk})^{\htt}$.
\end{itemize}
\end{lemma}

\begin{proof}
(1) Let $X\in\mathcal{M}\ast\mathcal{M}$. Then there exists an $\mathfrak{s}$-conflation $M\to X\to M'\dashrightarrow$ such that $M,M'\in\mathcal{M}$. By $\mathbb{E}(\mathcal{M},\mathcal{M})=0$, the $\mathfrak{s}$-conflation splits. 
Thus $X\cong M\oplus M'\in\mathcal{M}$.  

(2) By Lemma \ref{lem_perp}, we have $\mathcal{M}\subseteq{}^{\perp}\mathcal{M}= {}^{\perp}(\mathcal{M}^{\htt}_{n})$ for each non-negative integer $n$. 
Hence the former assertion holds. Moreover, the latter assertion follows from 
\begin{align}
\mathcal{M}^{\htt}_{n}=\cone(\mathcal{M}^{\htt}_{n-1},\mathcal{M})\subseteq\cone(\mathcal{M}^{\htt}_{n},\mathcal{M}).\notag
\end{align}

(3) It is similar to (2). 

(4) By (2) and Lemma \ref{lem_perp}, we obtain
\begin{align}
\mathcal{M}^{\htt}_{n}\subseteq({}^{\perp}(\mathcal{M}_{n}^{\htt}))^{\perp}\subseteq \mathcal{M}^{\perp}=(\mathcal{M}_{m}^{\chk})^{\perp}\notag
\end{align}
for all non-negative integers $n,m$.
Thus we have the assertion.

(5) By applying Proposition \ref{prop_chk}(1) to $\mathcal{X}=\mathcal{W}=\mathcal{M}$, the subcategory $\mathcal{M}^{\chk}_{m}$ is closed under extensions.
Thus, applying Lemma \ref{lem_htt-cone}(1-a) to $\mathcal{X}=\mathcal{M}^{\chk}_{m}$ and $\mathcal{W}=\mathcal{M}$ implies  $(\mathcal{M}^{\chk}_{m})^{\htt}_{n}=\cone (\mathcal{M}^{\htt}_{n-1}, \mathcal{M}^{\chk}_{m})$.
Similarly, we have $(\mathcal{M}^{\htt}_{n})^{\chk}_{m}=\cocone(\mathcal{M}^{\htt}_{n}, \mathcal{M}^{\chk}_{m-1})$. 
Thus the assertion follows from 
\begin{align}
\cocone( \mathcal{M}^{\htt}_{n}, \mathcal{M}^{\chk}_{m-1})
&=\cocone( \cone (\mathcal{M}^{\htt}_{n-1}, \mathcal{M}),\mathcal{M}^{\chk}_{m-1})&\textnormal{by definition}\notag \\
&=\cone ( \mathcal{M}^{\htt}_{n-1}, \cocone (\mathcal{M}, \mathcal{M}^{\chk}_{m-1}))&\textnormal{by Lemma \ref{lem_basic}(7)}\notag \\
&=\cone ( \mathcal{M}^{\htt}_{n-1}, \mathcal{M}^{\chk}_{m})&\textnormal{by definition}.\notag
\end{align}
The proof is complete.
\end{proof}

For presilting subcategories of $\mathcal{C}$, Proposition \ref{prop_htt} induces the following result.

\begin{proposition}\label{prop_presilt}
Let $\mathcal{M}$ be a presilting subcategory of $\mathcal{C}$.
Then the following statements hold.
\begin{itemize}
\item[(1)] $\mathcal{M}^{\htt}$ is the smallest subcategory containing $\mathcal{M}$ and closed under extensions, cones and direct summands. Moreover, if $\mathcal{M}$ is closed under cocones, then we have $\mathcal{M}^{\htt}=\thick\mathcal{M}$.
\item[(2)] $\mathcal{M}^{\chk}$ is the smallest subcategory containing $\mathcal{M}$ and closed under extensions, cocones and direct summands. Moreover, if $\mathcal{M}$ is closed under cones, then we have $\mathcal{M}^{\chk}=\thick\mathcal{M}$.
\item[(3)] $\mathcal{M}^{\tld}$ is the smallest subcategory containing $\mathcal{M}$ and closed under extensions, cones, cocones and direct summands, that is, $\mathcal{M}^{\tld}=\thick \mathcal{M}$.
\end{itemize}
\end{proposition}

\begin{proof}
(1) Let $\mathcal{M}$ be a presilting subcategory. 
Then $\mathcal{M}$ is closed under direct summands. 
By applying Proposition \ref{prop_htt}(3) to $\mathcal{X}=\mathcal{W}=\mathcal{M}$, the subcategory $\mathcal{M}^{\htt}$ is the  smallest subcategory containing $\mathcal{M}$ and closed under extensions, cones and direct summands. 
Moreover, if $\mathcal{M}$ is closed under cocones, then  $\mathcal{M}^{\htt}=\thick \mathcal{M}$ holds by Proposition \ref{prop_htt}(5). 

(2) It is similar to (1). 

(3) It follows from (1) that $\mathcal{M}^{\htt}$ is closed under extensions, cones and direct summands. 
By Lemma \ref{lem_selfort}(2), we can apply Proposition \ref{prop_chk}(5) to $\mathcal{X}=\mathcal{M}^{\htt}$ and $\mathcal{W}=\mathcal{M}$. 
Thus we obtain $\mathcal{M}^{\tld}=\thick(\mathcal{M}^{\htt})$. 
By Lemma \ref{lem_conecl}, we have $\mathcal{M}^{\htt}\subseteq\thick\mathcal{M}$, and hence $\mathcal{M}^{\tld}=\thick\mathcal{M}$.  
\end{proof}

\begin{lemma}\label{lem_thick-htt}
If $\mathcal{M}$ is presilting, then  $\thick \mathcal{M}\cap\mathcal{M}^{\perp}=\mathcal{M}^{\htt}$ and $\thick \mathcal{M}\cap{}^{\perp}\mathcal{M}=\mathcal{M}^{\chk}$.
\end{lemma}

\begin{proof}
Let $\mathcal{M}$ be a presilting subcategory. 
By Lemma \ref{lem_selfort}(3), we can apply Lemma \ref{lem_htt-cone}(2) to $\mathcal{X}=\mathcal{M}^{\chk}$ and $\mathcal{W}=\mathcal{M}$. 
Thus we have 
\begin{align}
\mathcal{M}^{\htt}_{n}=\cone(\mathcal{M}^{\htt}_{n-1}, \mathcal{M}^{\chk})\cap (\mathcal{M}^{\chk})^{\perp}=(\mathcal{M}^{\chk})^{\htt}_{n}\cap \mathcal{M}^{\perp},\notag
\end{align}
where the last equality follows from Lemmas \ref{lem_perp} and \ref{lem_selfort}(5). 
By Proposition \ref{prop_presilt}(3), we have $\thick \mathcal{M} \cap \mathcal{M}^{\perp}=\mathcal{M}^{\htt}$. 
Similarly, we obtain $\thick \mathcal{M}\cap{}^{\perp}\mathcal{M}=\mathcal{M}^{\chk}$. 
\end{proof}

\section{Bounded hereditary cotorsion pairs and silting subcategories}
In this section, we study a relationship between hereditary cotorsion pairs and silting subcategories in an extriangulated category $\mathcal{C}=(\mathcal{C},\mathbb{E},\mathfrak{s})$. 
Moreover, we show that our result recovers \cite[Corollary 5.9]{MSSS13} and \cite[Corollary 5.6]{AR91}. 
Namely, we have a bijection between bounded co-$t$-structures and silting subcategories for a triangulated category, and a bijection between contravariantly finite resolving subcategories and tilting modules for an artin algebra with finite global dimension.

\subsection{Silting subcategories}

We introduce the notion of silting subcategories in an extriangulated category, which is a generalization of silting subcategories in a triangulated category.

\begin{definition}\label{def_silting}
Let $\mathcal{C}$ be an extriangulated category and $\mathcal{M}$ a subcategory of $\mathcal{C}$. 
We call $\mathcal{M}$ a \emph{silting subcategory} of $\mathcal{C}$ if it satisfies the following conditions.
\begin{itemize}
\item[(1)] $\mathcal{M}$ is a presilting subcategory, that is, $\mathcal{M}=\add\mathcal{M}$ and $\mathbb{E}^{k}(\mathcal{M},\mathcal{M})=0$ for all $k\geq 1$.
\item[(2)] $\mathcal{C}=\thick \mathcal{M}$. 
\end{itemize}
Let $\silt \mathcal{C}$ denote the set of all silting subcategories in $\mathcal{C}$. An object $M\in \mathcal{C}$ is called a \emph{silting object} if $\add M$ is a silting subcategory of $\mathcal{C}$.
\end{definition}

We give an example of silting subcategories.
\begin{example}
Let $A$ be an artin algebra and let $\mathcal{P}^{\infty}(A)$ denote the category of finitely generated right $A$-modules of finite projective dimension. 
Since $\mathcal{P}^{\infty}(A)$ is closed under extensions, it is an extriangulated category. 
We can easily check that $\add A$ is a silting subcategory of $\mathcal{P}^{\infty}(A)$. 
\end{example}

We give an easy observation of silting subcategories. 

\begin{lemma}\label{lem_silt}
Let $\mathcal{M}\in \silt\mathcal{C}$. 
 If $\mathcal{N}$ is a self-orthogonal subcategory with $\mathcal{M}\subseteq \mathcal{N}$, then $\mathcal{M}=\mathcal{N}$.
\end{lemma}

\begin{proof}
Let $X \in \mathcal{N}$. 
Since $\mathcal{N}$ is a self-orthogonal subcategory with $\mathcal{M} \subseteq \mathcal{N}$, we obtain $\mathbb{E}^{k}(\mathcal{M}, X)=0$ for each $k \ge 1$.
By Lemma \ref{lem_thick-htt}, we have $X \in \mathcal{M}^{\htt}$. 
Thus there exists an $\mathfrak{s}$-conflation $K \to M \to X \dashrightarrow$ such that $K \in \mathcal{M}^{\htt}$ and $M \in \mathcal{M}$. 
By $X \in \mathcal{N}$ and $\mathcal{M}^{\htt}\subseteq \mathcal{N}^{\htt}$, it follows from Lemma \ref{lem_selfort}(4) that the $\mathfrak{s}$-conflation splits. 
Since $\mathcal{M}$ is closed under direct summands, we have $X \in \mathcal{M}$.
\end{proof}

Following \cite[Proposition 2.20]{AI12}, we give a sufficient condition for silting subcategories to admit additive generators.

\begin{proposition}\label{prop_siltobj}
Assume that $\mathcal{C}$ has a silting object.
Then each silting subcategory admits an additive generator. 
Moreover, if $\mathcal{C}$ is a Krull--Schmidt category, then $M\mapsto\add M$ gives a bijection between the set of isomorphisms classes of basic silting objects and the set of silting subcategories.
\end{proposition}

\begin{proof}
Let $\mathcal{N}$ be a silting subcategory of $\mathcal{C}$. 
Let $X \in \mathcal{N}^{\htt}_{n}$. 
By induction on $n$, we show that there exists $N_{X} \in \mathcal{N}$ such that $X \in \thick N_{X}$. 
If $n=0$, then this is clear. 
Let $n \ge 1$. 
Then we have an $\mathfrak{s}$-conflation $K \to N \to X \dashrightarrow$ with $K \in \mathcal{N}^{\htt}_{n-1}$ and $N \in \mathcal{N}$. 
By the induction hypothesis, there exists an object $N_{K} \in \mathcal{N}$ such that $K \in \thick N_{K}$. 
Put $N_{X}:=N_{K} \oplus N$. 
Then $X \in \thick N_{X}$. 
Similarly, for each $Y \in \mathcal{N}^{\chk}_{m}$, there exists $N_{Y} \in \mathcal{N}$ such that $Y \in \thick N_{Y}$. 

Let $M$ be a silting object of $\mathcal{C}$. 
By Proposition \ref{prop_presilt}(3), $M \in \mathcal{C}=\thick \mathcal{N}=\mathcal{N}^{\tld}$. 
Thus it follows from Lemma \ref{lem_selfort}(5) that there exists an $\mathfrak{s}$-conflation $U \to V \to M \dashrightarrow$ such that $U \in \mathcal{N}^{\htt}$ and $V \in \mathcal{N}^{\chk}$. 
By the first paragraph, we have objects $N_{U}, N_{V} \in \mathcal{N}$ with $U \in \thick N_{U}$ and $V \in \thick N_{V}$.
Let $\mathcal{N}':=\add (N_{U} \oplus N_{V})$. 
Then $\mathcal{N}'\subseteq\mathcal{N}$.
Moreover, it follows from $U, V \in \thick \mathcal{N}'$ that $M \in \thick \mathcal{N}'$. 
Since $M$ is a silting object, we have $\mathcal{N}' \in \silt \mathcal{C}$ with $\mathcal{N}' \subseteq \mathcal{N}$. 
By Lemma \ref{lem_silt}, $\mathcal{N}'=\mathcal{N}$, and hence we obtain the former assertion. 
We assume that $\mathcal{C}$ is a Krull--Schmidt category. 
Then $M \mapsto \add M$ gives a one-to-one correspondence between the set of isomorphism classes of basic objects and the set of subcategories of $\mathcal{C}$ containing additive generators. Thus the latter assertion holds. 
\end{proof} 

In the rest of this subsection, we show that if an artin algebra has finite global dimension, then silting objects coincide with tilting modules.
An object $P\in\mathcal{C}$ is said to be \emph{projective} if $\mathbb{E}(P,\mathcal{C})=0$.
Let $\proj \mathcal{C}$ denote the subcategory of $\mathcal{C}$ consisting of all projective objects in $\mathcal{C}$.
Dually we define injective objects and $\inj\mathcal{C}$.
Note that $\proj\mathcal{C}$ and $\inj\mathcal{C}$ are presilting subcategories.

\begin{proposition}\label{lem_silt-char}
The subcategory $\proj \mathcal{C}$ is a silting subcategory of $\mathcal{C}$ if and only if $\mathcal{C}=(\proj \mathcal{C})^{\htt}$. In this case, a subcategory $\mathcal{T}$ of $\mathcal{C}$ is a silting subcategory if and only if $\mathcal{T}$ satisfies the following conditions.
\begin{itemize}
\item[(1)] $\mathcal{T}$ is closed under direct summands.
\item[(2)] $\mathcal{T} \subseteq (\proj \mathcal{C})^{\htt}$.
\item[(3)] $\mathcal{T}$ is self-orthogonal.
\item[(4)] $\proj \mathcal{C} \subseteq \mathcal{T}^{\chk}$.
\end{itemize}
\end{proposition}

\begin{proof}
By definition, $\proj\mathcal{C}=\cocone(\proj\mathcal{C},\proj\mathcal{C})$, and hence $\proj\mathcal{C}$ is closed under cocones.
Thus it follows from Proposition \ref{prop_presilt}(1) that $(\proj\mathcal{C})^{\htt}=\thick(\proj\mathcal{C})$. Therefore we have the former assertion. 
We show the latter assertion.
Let $\mathcal{T}$ be a presilting subcategory of $\mathcal{C}$. 
We claim the ``only if'' part. 
By the former assertion and Lemma \ref{lem_thick-htt}, we have $\mathcal{T}\subseteq\mathcal{C}=(\proj \mathcal{C})^{\htt}$ and $\proj \mathcal{C} \subseteq {}^{\perp}\mathcal{T}=\mathcal{T}^{\chk}$ respectively. 
Hence (2) and (4) hold. 
We claim the ``if'' part. 
By the former assertion, (4) and Proposition \ref{prop_presilt}(3), we have $\mathcal{C}=(\proj\mathcal{C})^{\htt}\subseteq\mathcal{T}^{\tld}=\thick\mathcal{T}$.
This finishes the proof.
\end{proof}

Now we are ready to prove the following result.

\begin{corollary}\label{ex_fgl}
Let $A$ be an artin algebra and let $\mod A$ denote the category of finitely generated right $A$-modules. Then the following statements are equivalent. 
\begin{itemize}
\item[(1)] $A$ is a silting object of $\mod A$. 
\item[(2)] $A$ has finite global dimension. 
\item[(3)] Tilting $A$-modules of finite projective dimension coincide with silting objects of $\mod A$. 
\end{itemize}
\end{corollary}

\begin{proof}
This follows from Proposition \ref{lem_silt-char}.
\end{proof}

\subsection{Bijection between bounded hereditary cotorsion pairs and silting subcategories}
We say that a cotorsion pair $(\mathcal{X}, \mathcal{Y})$ is \emph{bounded} if $\mathcal{C}=\mathcal{X}^{\htt}$ and $\mathcal{C}=\mathcal{Y}^{\chk}$. 
Let $\bddhcotors \mathcal{C}$ denote the poset of bounded hereditary cotorsion pairs in $\mathcal{C}$. 
The following theorem is one of main results in this paper.

\begin{theorem}\label{thm1}
Let $\mathcal{C}$ be an extriangulated category. 
Then there exist mutually inverse bijections
\[
\begin{tikzcd}
\bddhcotors \mathcal{C} \rar[shift left, "\Phi"] & \silt \mathcal{C}, \lar[shift left, "\Psi"]
\end{tikzcd}
\]
where $\Phi(\mathcal{X},\mathcal{Y}):=\mathcal{X}\cap\mathcal{Y}$ and $\Psi(\mathcal{M}):=(\mathcal{M}^{\chk},\mathcal{M}^{\htt})$.
\end{theorem}

Note that if $\mathcal{M}$ is a silting subcategory of $\mathcal{C}$, then $\Psi(\mathcal{M})=({}^{\perp}\mathcal{M},\mathcal{M}^{\perp})$ by Lemma \ref{lem_thick-htt}.

\begin{remark} 
Theorem \ref{thm1} is not contained in \cite[Theorem 2]{ZZ20}.
Indeed, let $\mathcal{D}$ be the bounded homotopy category of finitely generated projective modules over a finite dimensional algebra.
Then silting subcategories are abundant in $\mathcal{D}$ and bijectively correspond to bounded co-$t$-structures on $\mathcal{D}$.
On the other hand, there are no other tilting subcategories (in the sense of \cite{ZZ20}) in a triangulated category except the zero subcategory.
\end{remark}

We show that $\Phi$ is well-defined. 

\begin{proposition}\label{prop_Phi}
If $(\mathcal{X}, \mathcal{Y})$ is a bounded hereditary cotorsion pair in $\mathcal{C}$, then $\mathcal{M}:=\mathcal{X} \cap \mathcal{Y} \in \silt \mathcal{C}$ satisfying $\mathcal{X}=\mathcal{M}^{\chk}$ and $\mathcal{Y}=\mathcal{M}^{\htt}$.  
\end{proposition}

\begin{proof}
Let $(\mathcal{X}, \mathcal{Y})$ be a bounded hereditary cotorsion pair and $\mathcal{M}:=\mathcal{X} \cap \mathcal{Y}$. 
We show $\mathcal{Y}=\mathcal{M}^{\htt}$. 
By $\mathcal{M}\subseteq \mathcal{Y}$, we have $\mathcal{M}^{\htt}\subseteq \mathcal{Y}^{\htt}=\mathcal{Y}$, where the last equality follows from Lemmas \ref{lem_cotors}(2) and \ref{lem_conecl}. 
We prove the converse inclusion. 
Since $\mathcal{X}$ is closed under extensions, we have $\mathcal{X} \subseteq \cocone (\mathcal{M}, \mathcal{X})$. 
Thus it follows from Lemma \ref{lem_htt-cone}(1-a) that
\begin{align}
\mathcal{C}=\mathcal{X}^{\htt}=\cone(\mathcal{M}^{\htt}, \mathcal{X}),\notag
\end{align}
where the first equality follows from the assumption that $(\mathcal{X},\mathcal{Y})$ is bounded. 
In particular, we obtain  $\mathcal{Y}\subseteq\cone(\mathcal{M}^{\htt},\mathcal{X})$. 
Since $\mathcal{M}^{\htt} \subseteq \mathcal{Y}$ and $\mathcal{Y}$ is closed under extensions, we have $\mathcal{Y} \subseteq \cone (\mathcal{M}^{\htt}, \mathcal{M})=\mathcal{M}^{\htt}$. 
Similarly, we obtain $\mathcal{X}=\mathcal{M}^{\chk}$. 
Since $\mathcal{M}$ is presilting, it follows from Proposition \ref{prop_presilt}(3) that $\mathcal{M}^{\tld}=\thick \mathcal{M}$.
Thus we have 
\begin{align}
\mathcal{C} =\mathcal{Y}^{\chk} =\mathcal{M}^{\tld} =\thick \mathcal{M}, \notag
\end{align}
where the first equality follows from the assumption that $(\mathcal{X}, \mathcal{Y})$ is bounded. 
The proof is complete.
\end{proof}

We show that $\Psi$ is well-defined.

\begin{proposition}\label{prop_Psi}
If $\mathcal{M} \in \silt \mathcal{C}$, then $(\mathcal{M}^{\chk},\mathcal{M}^{\htt})$ is a bounded hereditary cotorsion pair in $\mathcal{C}$ satisfying $\mathcal{M}=\mathcal{M}^{\chk} \cap \mathcal{M}^{\htt}$.  
\end{proposition}

\begin{proof}
Let $\mathcal{M}$ be a silting subcategory of $\mathcal{C}$. 
Then it follows from Proposition \ref{prop_presilt}(1) and (2) that $(\mathcal{M}^{\chk}, \mathcal{M}^{\htt})$ satisfies (CP1).
By Lemma \ref{lem_selfort}(4), the pair satisfies (CP2) and (HCP).
By $\mathcal{M} \in \silt \mathcal{C}$ and  Proposition \ref{prop_presilt}(3), we have $\mathcal{C}=\thick \mathcal{M}=\mathcal{M}^{\tld}$. 
Thus it follows from Lemma \ref{lem_selfort}(5) that
\begin{align}
\mathcal{C}=\mathcal{M}^{\tld}=(\mathcal{M}^{\htt})^{\chk}=\cocone(\mathcal{M}^{\htt},\mathcal{M}^{\chk})=\cone(\mathcal{M}^{\htt},\mathcal{M}^{\chk})=(\mathcal{M}^{\chk})^{\htt}.\notag
\end{align}
This implies that $(\mathcal{M}^{\chk},\mathcal{M}^{\htt})$ is a bounded hereditary cotorsion pair in $\mathcal{C}$.
Due to Lemma \ref{lem_selfort}(4), $\mathcal{M}^{\chk}\cap\mathcal{M}^{\htt}$ is a self-orthogonal subcategory.
Since $\mathcal{M}\in \silt \mathcal{C}$ and $\mathcal{M}\subseteq \mathcal{M}^{\chk} \cap \mathcal{M}^{\htt}$, it follows from Lemma \ref{lem_silt} that $\mathcal{M}=\mathcal{M}^{\chk} \cap \mathcal{M}^{\htt}$.
\end{proof}

Now we are ready to prove Theorem \ref{thm1}. 

\begin{proof}[Proof of Theorem \ref{thm1}]
This follows from Propositions \ref{prop_Phi} and \ref{prop_Psi}.
\end{proof}

Let $\mathcal{D}$ be a triangulated category with shift functor $\Sigma$.
A co-$t$-structure $(\mathcal{X}, \mathcal{Y})$ on $\mathcal{D}$ is said to be \emph{bounded} if $\displaystyle \cup_{n \in \mathbb{Z}} \Sigma^{n} \mathcal{X}= \mathcal{D}=\cup_{n \in \mathbb{Z}} \Sigma^{n} \mathcal{Y}$. 
Let $\bddcotstr \mathcal{D}$ denote the poset of bounded co-$t$-structures on $\mathcal{D}$.
By Theorem \ref{thm1}, we can recover the following result. 

\begin{corollary}[{\cite[Corollary 5.9]{MSSS13}}]
Let $\mathcal{D}$ be a triangulated category. 
Then there exist mutually inverse bijections 
\[
\begin{tikzcd}
\bddcotstr \mathcal{D} \rar[shift left, "\Phi"] & \silt \mathcal{D}, \lar[shift left, "\Psi"]
\end{tikzcd}
\]
where $\Phi(\mathcal{X},\mathcal{Y}):=\mathcal{X}\cap\mathcal{Y}$ and $\Psi(\mathcal{M}):=(\mathcal{M}^{\chk},\mathcal{M}^{\htt})$.
\end{corollary}

\begin{proof}
By Example \ref{ex_scotors}(1) and Remark \ref{rem_filt-hat}, we have $\bddhcotors \mathcal{D}=\bddcotstr \mathcal{D}$.
Thus the assertion follows from Theorem \ref{thm1}. 
\end{proof}

By the correspondence in Theorem \ref{thm1}, we introduce a partial order on silting subcategories.
For two subcategories $\mathcal{M},\mathcal{N}$ of $\mathcal{C}$, we write $\mathcal{M}\geq \mathcal{N}$ if $\mathbb{E}^{k}(\mathcal{M},\mathcal{N})=0$ for each $k \ge 1$.
\begin{proposition}
For $\mathcal{M},\mathcal{N}\in \silt\mathcal{C}$, the following statements are equivalent.
\begin{itemize}
\item[(1)] $\mathcal{M}\geq \mathcal{N}$. 
\item[(2)] $\mathcal{M}^{\htt}\supseteq \mathcal{N}^{\htt}$.
\item[(3)] $\mathcal{M}^{\chk}\subseteq \mathcal{N}^{\chk}$.
\end{itemize}
In particular, $\geq$ gives a partial order on $\silt\mathcal{C}$.
\end{proposition}

\begin{proof}
We only prove (1)$\Leftrightarrow$(2) since  the proof of (1)$\Leftrightarrow$(3) is similar.

(1)$\Rightarrow$(2): By Lemma \ref{lem_perp}, $\mathcal{M}\geq \mathcal{N}$ implies $\mathbb{E}^{k}(\mathcal{M},\mathcal{N}^{\htt})=0$ for each $k\ge 1$.
Hence $\mathcal{N}^{\htt}\subseteq \mathcal{M}^{\perp}$. The assertion follows from Lemma \ref{lem_thick-htt}.

(2)$\Rightarrow$(1): It follows from (2) and Lemma \ref{lem_thick-htt} that $\mathcal{N}\subseteq \mathcal{N}^{\htt}\subseteq \mathcal{M}^{\htt}=\mathcal{M}^{\perp}$. 
Hence we have the assertion.
\end{proof}

In the following, we explain that Theorem \ref{thm1} can recover Auslander--Reiten's result (see Corollary \ref{cor_AR}). 
Following \cite{Kr21}, we introduce the notion of resolving subcategories.
Let $\mathcal{X}$ be a subcategory of $\mathcal{C}$.
We call $\mathcal{X}$ a \emph{resolving subcategory} of $\mathcal{C}$ if $\mathcal{C}=\cone(\mathcal{C},\mathcal{X})$ and it is closed under extensions, cocones and direct summands.
We can easily check that if $\mathcal{C}$ has enough projective objects (i.e., $\mathcal{C}=\cone(\mathcal{C},\proj\mathcal{C})$), then $\mathcal{X}$ is a resolving subcategory if and only if it contains all projective objects of $\mathcal{C}$ and it is closed under extensions, cocones and direct summands. 
The subcategory $\mathcal{X}$ is said to be \emph{contravariantly finite} if each $M \in \mathcal{C}$ admits a right $\mathcal{X}$-approximation. 
Dually, we define coresolving subcategories and covariantly finite subcategories.

In the rest of this subsection, we assume that $\mathcal{C}$ is a Krull--Schmidt category and satisfies the following condition introduced in \cite[Condition 5.8]{NP19}. 

\begin{condition}[WIC]
Let $h=gf$ be morphisms in an extriangulated category $\mathcal{C}$. 
If $h$ is an $\mathfrak{s}$-inflation, then so is $f$. 
Dually, if $h$ is an $\mathfrak{s}$-deflation, then so is $g$. \end{condition}

We have the following Wakamatsu-type lemma. 

\begin{lemma}\label{lem_wak}
Assume that $\mathcal{C}$ has enough projective objects. 
Let $\mathcal{X}$ be a contravariantly finite resolving subcategory of $\mathcal{C}$. 
Then $\mathcal{C}=\cone (\mathcal{X}^{\perp}, \mathcal{X})$.
\end{lemma}

\begin{proof}
By the dual statement of \cite[Lemma 3.1 and Remark 3.2]{LZ}, we have $\mathcal{C}=\cone(\mathcal{X}^{\perp_{1}},\mathcal{X})$.
Thus it is enough to prove $\mathcal{X}^{\perp_{1}}=\mathcal{X}^{\perp}$.
Since $\mathcal{X}^{\perp_{1}}\supseteq \mathcal{X}^{\perp}$ is clear, we show the converse inclusion. 
Let $X \in \mathcal{X}$. 
Then there exists an $\mathfrak{s}$-conflation $K_{1} \to P_{0} \to X \dashrightarrow$ such that $P_{0} \in \proj \mathcal{C}$. 
Since $\mathcal{X}$ is closed under cocones, we have $K_{1} \in \mathcal{X}$.
Repeating this argument, for each $i\geq 1$, we have an $\mathfrak{s}$-conflation $K_{i+1} \to P_{i} \to K_{i}\dashrightarrow$ with $P_{i} \in \proj \mathcal{C}$ and $K_{i+1}, K_{i} \in \mathcal{X}$.
For each $M \in \mathcal{X}^{\perp_{1}}$ and $k \ge 1$, we have an isomorphism $\mathbb{E}^{k+1}(X, M) \cong \mathbb{E}(K_{k}, M)$. 
By $K_{k}\in\mathcal{X}$ and $M \in \mathcal{X}^{\perp_{1}}$, the assertion holds. This finishes the proof.
\end{proof}

Let $\hcotors \mathcal{C}$ denote the set of hereditary cotorsion pairs in $\mathcal{C}$. 
Following \cite[$\S$3]{AR91}, we give a connection between contravariantly finite resolving subcategories, covariantly finite coresolving subcategories and hereditary cotorsion pairs. 

\begin{proposition}\label{prop_res-hcotors}
Assume that $\mathcal{C}$ has enough projective objects and enough injective objects. 
Then there exist mutually inverse bijections 
\[
\begin{tikzcd}
\{ \mathrm{contravariantly\;finite\;resolving\;subcategories\;of\;}\mathcal{C} \} \dar[shift right, "F_{1}"']\arrow[dd,xshift=-20mm,"F"']\\
\hcotors\mathcal{C} \uar[shift right, "G_{1}"'] \dar[shift right, "F_{2}"']\\
\{ \mathrm{covariantly\;finite\;coresolving\;subcategories\;of\;}\mathcal{C}\}  \uar[shift right, "G_{2}"'] \arrow[uu,xshift=20mm,"G"'],
\end{tikzcd}
\]
where $F(\mathcal{X}):=\mathcal{X}^{\perp}$,  $G(\mathcal{Y}):={}^{\perp}\mathcal{Y}$, $F_{1}(\mathcal{X}):=(\mathcal{X}, \mathcal{X}^{\perp})$, $G_{1}(\mathcal{X}, \mathcal{Y}):=\mathcal{X}$, $F_{2}(\mathcal{X}, \mathcal{Y}):=\mathcal{Y}$ and $G_{2}(\mathcal{Y}):=({}^{\perp}\mathcal{Y}, \mathcal{Y})$.
\end{proposition}

\begin{proof}
First we show that $F$ and $G$ are mutually inverse bijections. 
Let $\mathcal{X}$ be a contravariantly finite resolving subcategory. 
Clearly $\mathcal{X}^{\perp}$ is a coresolving subcategory. 
Since 
\begin{align}
\mathcal{C}
&=\cocone(\inj \mathcal{C},\mathcal{C})&\textnormal{since $\mathcal{C}$ has enough injective objects}\notag\\
&=\cocone(\inj \mathcal{C}, \cone (\mathcal{X}^{\perp}, \mathcal{X}))&\textnormal{by Lemma \ref{lem_wak}} \notag\\
&\subseteq \cocone(\mathcal{X}^{\perp}\ast \inj \mathcal{C}, \mathcal{X})&\textnormal{by Lemma \ref{lem_basic}(4)}\notag\\
&\subseteq \cocone(\mathcal{X}^{\perp}, \mathcal{X})&\textnormal{since $\mathcal{X}^{\perp}$ is closed under extensions,}\notag
\end{align}
the subcategory $\mathcal{X}^{\perp}$ is covariantly finite. 
Hence $F$ is well-defined.
Similarly, $G$ is well-defined.
We prove $\mathcal{X} ={}^{\perp}(\mathcal{X}^{\perp})$. 
Since $\mathcal{X}\subseteq{}^{\perp}(\mathcal{X}^{\perp})$ clearly holds, we show the converse inclusion. 
Let $M \in {}^{\perp}(\mathcal{X}^{\perp})$.
By Lemma \ref{lem_wak}, we have an $\mathfrak{s}$-conflation $Y \to X \to M \dashrightarrow$ with $Y \in \mathcal{X}^{\perp}$ and $X \in \mathcal{X}$.
Since the $\mathfrak{s}$-conflation splits, we have $M \in \mathcal{X}$.
Thus $GF=1$ holds. 
Similarly, we obtain $FG=1$. 
Hence the assertion holds. 

Next we show that $F_{1}$ and $G_{1}$ are mutually inverse bijections. 
Let $\mathcal{X}$ be a contravariantly finite resolving subcategory.
By definition, $\mathbb{E}^{k}(\mathcal{X},\mathcal{X}^{\perp})=0$ for all $k\geq 1$.
It follows from Lemma \ref{lem_wak} that  $\mathcal{C}=\cone(\mathcal{X}^{\perp}, \mathcal{X})$ holds. 
By the first paragraph, $\mathcal{X}^{\perp}$ is a covariantly finite coresolving subcategory and $\mathcal{X} ={}^{\perp}(\mathcal{X}^{\perp})$. 
Applying the dual statement of Lemma \ref{lem_wak} to $\mathcal{X}^{\perp}$, we have $\mathcal{C}=\cocone (\mathcal{X}^{\perp}, {}^{\perp}(\mathcal{X}^{\perp}))=\cocone (\mathcal{X}^{\perp}, \mathcal{X})$. 
Hence $F_{1}$ is well-defined. 
We prove that $G_{1}$ is well-defined. 
Let $(\mathcal{X}, \mathcal{Y})$ be a hereditary cotorsion pair in $\mathcal{C}$.
Then $\mathcal{X}$ is a resolving subcategory by Lemma \ref{lem_cotors}.
Let $M \in \mathcal{C}$. 
Since $\mathcal{C}=\cone(\mathcal{Y},\mathcal{X})$, we have an $\mathfrak{s}$-conflation $Y \to X \xrightarrow{f} M \dashrightarrow$ with $Y \in \mathcal{Y}$ and $X \in \mathcal{X}$. 
By $\mathbb{E}(\mathcal{X}, \mathcal{Y})=0$, the morphism $f$ is a right $\mathcal{X}$-approximation of $M$.
Hence $\mathcal{X}$ is a contravariantly finite resolving subcategory.
Clearly $G_{1}F_{1}=1$. Moreover, by Lemma \ref{lem_cotors}(1), we have $F_{1}G_{1}=1$. 
Similarly, $F_{2}$ and $G_{2}$ are mutually inverse bijections.
The proof is complete.
\end{proof}

For a subcategory $\mathcal{X}$ of $\mathcal{C}$, we say that the projective dimension $\pd \mathcal{X}$ of $\mathcal{X}$ is at most $n$ if $\mathcal{X} \subseteq (\proj \mathcal{C})^{\htt}_{n}$.
Dually, we define the injective dimension $\id \mathcal{X}$ of $\mathcal{X}$.
If $\mathcal{C}$ has enough projective objects, then we give a characterization for $\pd \mathcal{X}$ to be at most $n$. 

\begin{lemma}\label{lem_pd-ext}
Let $\mathcal{C}$ be an extriangulated category and $\mathcal{X}$ a subcategory of $\mathcal{C}$. Fix a non-negative integer $n$.
If $\pd\mathcal{X}\leq n$, then we have
\begin{align}
\mathcal{C}=\mathcal{X}^{\perp_{>n}}:=\{M \in \mathcal{C}\mid \mathbb{E}^{k}(\mathcal{X}, M)=0\textnormal{ for each $k\ge n+1$}\}.\notag
\end{align}
Moreover, if $\mathcal{C}$ has enough projective objects, then the converse also holds.
\end{lemma}

\begin{proof}
Clearly, we have $\mathcal{C}=(\proj \mathcal{C})^{\perp}$. 
Since 
\begin{align}
(\proj \mathcal{C})^{\htt}=\cone((\proj \mathcal{C})^{\htt},\proj \mathcal{C})\subseteq\cone((\proj \mathcal{C})^{\perp}, \proj \mathcal{C}), \notag
\end{align}
it follows from Lemma \ref{lem_cldir}(1) that $(\proj \mathcal{C})^{\htt}_{n}=\{X \in (\proj \mathcal{C})^{\htt}\mid \mathbb{E}^{k}(X, \mathcal{C})=0\textnormal{\; for\;all\;}k \ge n+1\}$. 
Assume $\pd \mathcal{X}\le n$. 
Let $M \in \mathcal{C}$. 
Since $\mathcal{X}\subseteq(\proj \mathcal{C})^{\htt}_{n}$, we have $\mathbb{E}^{k}(\mathcal{X},M)=0$ for each $k\ge n+1$. 
Hence the former assertion holds. 
We assume that $\mathcal{C}$ has enough projective objects. 
Then, for each $X \in \mathcal{X}$, there exists an  $\mathfrak{s}$-conflation $K_{1} \to P_{0} \to X \dashrightarrow$ such that $P_{0} \in \proj \mathcal{C}$. 
Inductively, we have an $\mathfrak{s}$-conflation $K_{i+1} \to P_{i} \to K_{i}\dashrightarrow$ with $P_{i} \in \proj \mathcal{C}$ for each $1 \le i \le n$.
Applying $\mathcal{C}(-, K_{n+1})$ to the $\mathfrak{s}$-conflations gives an isomorphism $\mathbb{E}(K_{n}, K_{n+1}) \cong \mathbb{E}^{n+1}(X, K_{n+1})$. 
By $X \in \mathcal{X}$ and $K_{n+1} \in \mathcal{C}=\mathcal{X}^{\perp_{>n}}$, we have $\mathbb{E}(K_{n+1}, K_{n})=0$. 
Since the $\mathfrak{s}$-conflation $K_{n+1} \to P_{n} \to K_{n} \dashrightarrow$ splits, we obtain $K_{n} \in \proj \mathcal{C}$.
Thus $\pd X \le n$ holds.
\end{proof}

By Theorem \ref{thm1} and Proposition \ref{prop_res-hcotors}, we obtain a relationship between contravariantly finite resolving subcategories, covariantly finite coresolving subcategories and silting subcategories. 

\begin{theorem}\label{thm3}
Let $\mathcal{C}$ be a Krull--Schmidt extriangulated category satisfying the condition (WIC).
Assume that $\mathcal{C}$ has enough projective objects and enough  injective objects.
Then there exist mutually inverse bijections 
\[
\begin{tikzcd}
\{\textnormal{$\mathcal{X}$: contravariantly finite resolving subcategory of $\mathcal{C}$ $\mid$ $\mathcal{C}=\mathcal{X}^{\htt}$, $\pd \mathcal{X}<\infty$}\} \dar[shift right, "\Phi_{1}"']\\
\silt \mathcal{C} \uar[shift right, "\Psi_{1}"'] \dar[shift right, "\Phi_{2}"']\\
\{\textnormal{$\mathcal{Y}$: covariantly finite coresolving subcategory of  $\mathcal{C}$ $\mid$ $\mathcal{C}=\mathcal{Y}^{\chk}$, $\id \mathcal{Y}<\infty$}\}  \uar[shift right, "\Psi_{2}"'],
\end{tikzcd}
\]
where $\Phi_{1}(\mathcal{X}):=\mathcal{X} \cap \mathcal{X}^{\perp}$,  $\Psi_{1}(\mathcal{M}):={}^{\perp}\mathcal{M}$, $\Phi_{2}(\mathcal{M}):=\mathcal{M}^{\perp}$ and $\Psi_{2}(\mathcal{Y}):={}^{\perp}\mathcal{Y} \cap \mathcal{Y}$.
\end{theorem}

\begin{proof}
We only prove that $\Phi_{1}$ and $\Psi_{1}$ are mutually inverse bijections since the proof for $\Phi_{2}$ and $\Psi_{2}$ is similar. 
By Theorem \ref{thm1} and Proposition \ref{prop_res-hcotors}, it is enough to show
\begin{align}
\{ (\mathcal{X},\mathcal{Y})\in\hcotors\mathcal{C}\mid \mathcal{C}=\mathcal{X}^{\htt},\ \pd\mathcal{X}<\infty\}=\bddhcotors\mathcal{C},\notag
\end{align}
or equivalently, $\pd\mathcal{X}<\infty$ if and only if $\mathcal{C}=\mathcal{Y}^{\chk}$.
Let $(\mathcal{X},\mathcal{Y})$ be a hereditary cotorsion pair.
We claim $\mathcal{Y}^{\chk}_{n}=\mathcal{X}^{\perp_{>n}}$.
Let $M \in \mathcal{X}^{\perp_{>n}}$.
Since $\mathcal{C}$ has enough injective objects, there exists an $\mathfrak{s}$-conflation $M \to I^{0} \to C^{1} \dashrightarrow$ with $I^{0} \in \inj \mathcal{C}$. 
Inductively, we have an $\mathfrak{s}$-conflation $C^{i} \to I^{i} \to C^{i+1} \dashrightarrow$ with $I^{i} \in \inj \mathcal{C}$ for $1 \le i \le n-1$. 
Applying $\mathcal{C}(\mathcal{X},-)$ to the $\mathfrak{s}$-conflations gives an isomorphism $\mathbb{E}(\mathcal{X}, C^{n}) \cong \mathbb{E}^{n+1}(\mathcal{X}, M)$. 
By $M \in \mathcal{X}^{\perp_{>n}}$ and Lemma \ref{lem_cotors}(1), we have $C^{n} \in \mathcal{X}^{\perp_{1}}=\mathcal{Y}$. 
This implies $M\in \mathcal{Y}^{\chk}_{n}$.
Conversely, we show $\mathcal{Y}^{\chk}_{n} \subseteq \mathcal{X}^{\perp_{>n}}$ by induction on $n$. 
If $n=0$, then the assertion clearly holds. 
Let $n \ge 1$. 
By the induction hypothesis, we have
\begin{align}
 \mathcal{Y}^{\chk}_{n}=\cocone (\mathcal{Y}, \mathcal{Y}^{\chk}_{n-1})\subseteq \cocone (\mathcal{X}^{\perp_{>0}}, \mathcal{X}^{\perp_{>n-1}})\subseteq \mathcal{X}^{\perp_{>n}}.   \notag
\end{align}
Thus we have $\mathcal{Y}^{\chk}_{n}=\mathcal{X}^{\perp_{>n}}$.
By Lemma \ref{lem_pd-ext}, $\pd\mathcal{X}\le n$ if and only if $\mathcal{C}=\mathcal{Y}^{\chk}_{n}$.
The proof is complete.
\end{proof}

Now we are ready to prove the following result.

\begin{corollary}[{\cite[Corollary 5.6]{AR91}}]\label{cor_AR}
Let $A$ be an artin algebra with finite global dimension. 
Then $T \mapsto {}^{\perp}T$ gives a bijection between the set of isomorphism classes of basic tilting modules and the set of contravariantly finite resolving subcategories, and $T \mapsto T^{\perp}$ gives a bijection between the set of isomorphism classes of basic tilting modules and the set of covariantly finite coresolving subcategories.
\end{corollary}

\begin{proof}
By Proposition \ref{prop_siltobj} and Corollary \ref{ex_fgl}, $T \mapsto \add T$ gives a bijection between the set of isomorphism classes of basic tilting $A$-modules and the set of silting subcategories of $\mod A$. 
On the other hand, since $A$ is of finite global dimension, each contravariantly finite resolving subcategory $\mathcal{X}$ of $\mod A$ always satisfies $\mod A=(\add A)^{\htt} \subseteq \mathcal{X}^{\htt}$ and $\pd \mathcal{X}<\infty$.
Hence the first assertion follows from Theorem \ref{thm3}.
Similarly, we have the second assertion.
\end{proof}

\subsection{Silting subcategories and left Frobenius pairs}
Recently, Tan--Gao (\cite{TG}) and Ma--Liu--Hu--Geng (\cite{MLHG}) gave a connection between hereditary cotorsion pairs and left Frobenius pairs. 
Their result induces another proof of Theorem \ref{thm1}.
We recall the definition of left Frobenius pairs (see \cite{BMPS19, TG, MLHG} for details).

\begin{definition}
A pair $(\mathcal{X},\omega)$ of subcategories of $\mathcal{C}$ is called a \emph{left Frobenius pair} if it satisfies the following conditions.
\begin{itemize}
\item[(1)] $\mathcal{X}$ is closed under extensions, cocones and direct summands.
\item[(2)] $\omega$ is closed under direct summands and satisfies $\omega \subseteq \mathcal{X} \cap \mathcal{X}^{\perp}$ and $\mathcal{X}\subseteq\cocone(\omega, \mathcal{X})$.
\end{itemize}
\end{definition}

Silting subcategories are closely related to left Frobenius pairs. 
For a presilting subcategory $\mathcal{M}$, the pair $(\mathcal{M}^{\chk},\mathcal{M})$ is a left Frobenius pair by Lemma \ref{lem_selfort}(3) and Proposition \ref{prop_presilt}(2).
Conversely, for a left Frobenius pair $(\mathcal{X},\omega)$, the subcategory $\omega$ is clearly a presilting subcategory. 
Thus we have mutually inverse bijections 
\begin{equation}\label{map_silt-lfp}
\begin{tikzcd}
\{\textnormal{$\mathcal{M}$: presilting subcategory of $\mathcal{C}$}\}\rar[shift left, "\varphi"]&\{\textnormal{$(\mathcal{X},\omega)$: left Frobenius pair in $\mathcal{C}$} \mid \mathcal{X}=\omega^{\chk}\}\lar[shift left, "\psi"]
\end{tikzcd}
\end{equation}
given by $\varphi(\mathcal{M}):=(\mathcal{M}^{\chk},\mathcal{M})$ and $\psi(\mathcal{X},\omega):=\omega$. 
By restricting these bijections, we have the following result. 

\begin{proposition}\label{lem_frob}
Let $\mathcal{C}$ be an extriangulated category.
Then the maps $\varphi$ and $\psi$ in \eqref{map_silt-lfp} give mutually inverse bijections 
\[
\begin{tikzcd}
\silt\mathcal{C}\rar[shift left, "\varphi"]&\{\textnormal{$(\mathcal{X},\omega)$: left Frobenius pair in $\mathcal{C}$ $\mid$  $\mathcal{C}=\omega^{\tld}$}\}\lar[shift left, "\psi"].
\end{tikzcd}
\]
\end{proposition}

\begin{proof}
Let $(\mathcal{X},\omega)$ be a left Frobenius pair with $\mathcal{C}=\omega^{\tld}$.
We show $\mathcal{X}=\omega^{\chk}$. 
Since $\mathcal{X}$ is closed under cocones and $\omega\subseteq \mathcal{X}$, we have $\omega^{\chk}\subseteq \mathcal{X}$ by Lemma \ref{lem_conecl}.
By Proposition \ref{prop_presilt}(3), we have $\mathcal{X}\subseteq\mathcal{C}=\omega^{\tld}=\thick\omega$.
Thus it follows from Lemma \ref{lem_thick-htt} that $\mathcal{X} \subseteq  \thick \omega\cap{}^{\perp}\omega=\omega^{\chk}$. 
Hence we obtain $\mathcal{X}=\omega^{\chk}$.
Since the maps $\varphi$ and $\psi$ are well-defined by Proposition \ref{prop_presilt}(3), the assertion follows from \eqref{map_silt-lfp}.
\end{proof}

By \cite{MLHG}, we have the following result.

\begin{proposition}[{\cite[Theorem 3.12]{MLHG}}]\label{prop_frob}
Let $\mathcal{C}$ be an extriangulated category.
Then there exist mutually inverse bijections 
\[
\begin{tikzcd}
\{\textnormal{$(\mathcal{X},\mathcal{Y})\in \cotors(\thick \mathcal{X})$ $\mid$ $\mathbb{E}^{k}(\mathcal{X},\mathcal{Y})=0$ for all $k\geq 1$}\} \dar[shift right, "\varphi"']\\
\{\textnormal{$(\mathcal{X},\omega)$: left Frobenius pair in $\mathcal{C}$}\}\uar[shift right, "\psi"'],
\end{tikzcd}
\]
where $\varphi(\mathcal{X},\mathcal{Y}):=(\mathcal{X},\mathcal{X}\cap\mathcal{Y})$ and $\psi(\mathcal{X},\omega):=(\mathcal{X},\omega^{\htt})$.
\end{proposition}

For the convenience of the readers, we give a proof. 

\begin{proof}
First we show that $\varphi$ is well-defined.
Let $(\mathcal{X}, \mathcal{Y}) \in \cotors (\thick \mathcal{X})$ with $\mathbb{E}^{k}(\mathcal{X}, \mathcal{Y})=0$ for each $k \ge1$. 
Put $\omega:=\mathcal{X}\cap \mathcal{Y}$. 
Then $\mathcal{X}$ and $\omega$ are closed under direct summands and $\omega\subseteq\mathcal{X}\cap\mathcal{X}^{\perp}$.
Moreover, we can easily check $\mathcal{X}=\thick\mathcal{X}\cap{}^{\perp}\mathcal{Y}$.
Thus $\mathcal{X}$ is closed under extensions and cocones.
Since $\mathcal{X}$ is closed under extensions and $\mathcal{X}\subseteq\cocone(\mathcal{Y},\mathcal{X})$, we have $\mathcal{X}\subseteq\cocone(\omega,\mathcal{X})$. 
Hence $(\mathcal{X},\omega)$ is a left Frobenius pair. 
Next we show that $\psi$ is well-defined. 
Let $(\mathcal{X}, \omega)$ be a left Frobenius pair in $\mathcal{C}$. 
Then $\mathcal{X}, \omega^{\htt} \subseteq \thick \mathcal{X}$. 
By Proposition \ref{prop_presilt}(1), the pair $(\mathcal{X},\omega^{\htt})$ satisfies (CP1).
It follows from Lemma \ref{lem_perp} that  $\mathbb{E}^{k}(\mathcal{X}, \omega^{\htt})=0$ for each $k \ge 1$.
By Lemma \ref{lem_htt-cone}(1), the pair satisfies (CP3) and (CP4).
Hence $\psi$ is well-defined.  
By Lemma \ref{lem_cotors}(1), we have $\psi\varphi=1$. 
Moreover, it follows from Lemma \ref{lem_htt-cone}(2) that
\begin{align}
\omega =\omega^{\htt}_{0}=\cone(\omega^{\htt}_{-1},\mathcal{X})\cap \omega^{\htt}=\mathcal{X}\cap \omega^{\htt},\notag
\end{align}
and hence $\varphi\psi=1$.
\end{proof}

Now we are ready to reprove Theorem \ref{thm1}. 

\begin{proof}[Proof of Theorem \ref{thm1}]
We show
\begin{align}
\bddhcotors \mathcal{C}=\{\textnormal{$(\mathcal{X},\mathcal{Y})\in \cotors(\thick \mathcal{X})$ $\mid$ $\mathbb{E}^{k}(\mathcal{X},\mathcal{Y})=0$ for all $k\geq 1$, $\mathcal{C}=(\mathcal{X}\cap\mathcal{Y})^{\tld}$}\}.\notag
\end{align}
Let $(\mathcal{X}, \mathcal{Y})$ be a cotorsion pair in $\thick \mathcal{X}$ satisfying $\mathbb{E}^{k}(\mathcal{X}, \mathcal{Y})=0$ for each $k\ge1$. 
Put $\mathcal{W}:=\mathcal{X}\cap\mathcal{Y}$. 
By Proposition \ref{prop_frob}, $(\mathcal{X},\mathcal{W})$ is a left Frobenius pair. 
Thus it follows from Proposition \ref{prop_htt}(5) that $\mathcal{X}^{\htt}=\thick \mathcal{X}$. 
We assume $\mathcal{C}=\mathcal{W}^{\tld}$. 
By Lemma \ref{lem_conecl}, we have $\mathcal{W}^{\chk}\subseteq \mathcal{X}$, and hence $\mathcal{C}=\mathcal{X}^{\htt}$. 
Similarly, we obtain $\mathcal{C}=\mathcal{Y}^{\chk}$. 
Hence $(\mathcal{X}, \mathcal{Y})$ is a bounded hereditary cotorsion pair in $\mathcal{C}$. 
Conversely, let $(\mathcal{X}, \mathcal{Y}) \in \bddhcotors\mathcal{C}$ and $\mathcal{W}:=\mathcal{X}\cap\mathcal{Y}$.
Then $\mathcal{X}$ is closed under direct summands and $\mathcal{W}\subseteq\mathcal{X}\cap\mathcal{X}^{\perp}$.
By Lemma \ref{lem_cotors}, $\mathcal{X}$ is closed under extensions and cocones.
Thus it follows from Proposition \ref{prop_htt}(5) that $\mathcal{X}^{\htt}=\thick\mathcal{X}$.
Moreover, by Proposition \ref{prop_Phi}, we obtain $\mathcal{X}=\mathcal{W}^{\chk}$, and hence $\mathcal{C}=\mathcal{X}^{\htt}=\mathcal{W}^{\tld}$.
Therefore $(\mathcal{X}, \mathcal{Y})$ is a cotorsion pair in $\thick\mathcal{X}$ satisfying $\mathbb{E}^{k}(\mathcal{X}, \mathcal{Y})=0$ for each $k\ge1$ and $\mathcal{C}=\mathcal{W}^{\tld}$. 
Thus, by restricting the bijections in Proposition \ref{prop_frob}, we have 
\[
\begin{tikzcd}
\bddhcotors \mathcal{C} \rar[shift left, "\varphi"]&\{\textnormal{$(\mathcal{X},\omega)$: left Frobenius pair $\mid$  $\mathcal{C}=\omega^{\tld}$}\}.\lar[shift left, "\psi"]
\end{tikzcd}
\]
Therefore the assertion follows from Proposition \ref{lem_frob}.
\end{proof}


\begin{thebibliography}{99}

\bibitem[AET]{AET} T.~Adachi, H.~Enomoto, M.~Tsukamoto, \emph{Intervals of $s$-torsion pairs in extriangulated categories with negative first extensions}, arXiv:2103.09549.

\bibitem[AI]{AI12} T.~Aihara, O.~Iyama, \emph{Silting mutation in triangulated categories}, J. Lond. Math. Soc. (2) {\bf 85} (2012), no.~3, 633--668.

\bibitem[AB]{AB89} M.~Auslander, R.-O.~Buchweitz, \emph{The Homological theory of maximal Cohen--Macaulay approximations}, Colloque en l'honneur de Pierre Samuel (Orsay, 1987), Mem. Soc. Math. France, \textbf{38} (1989), 5--37.

\bibitem[AR]{AR91} M.~Auslander, I.~Reiten, \emph{Applications of contravariantly finite subcategories}, Adv. Math. \textbf{86} (1991), no.~1, 111--152.

\bibitem[BMPS]{BMPS19} V.~ Becerril, O.~Mendoza, M.~A.~P\'erez, V.~Santiago, \emph{Frobenius pairs in abelian categories}, J. Homotopy Relat. Struct., \textbf{14}(1)(2019), 1--50.

\bibitem[BBD]{BBD81} A.~A.~Be\u{\i}linson, J.~Bernstein, P.~Deligne, \emph{Faisceaux pervers}, Analysis and topology on singular spaces, I (Luminy, 1981), 5--171, Ast$\acute{\mathrm{e}}$risque, \textbf{100}, Soc. Math. France, Paris, 1982.

\bibitem[Bo]{Bo10}
M.~V.~Bondarko, \emph{Weight structures vs. t-structures; weight filtrations, spectral sequences, and complexes (for motives and in general)},  J. K-Theory 6 (2010), no.~3, 387--504. 

\bibitem[ET]{ET01} 
P.~Eklof, J.~Trlifaj, \emph{How to make Ext vanish}, Bull. London Math. Soc. {\bf 33} (2001), no.~1, 41--51. 

\bibitem[GNP]{GNP}
M.~Gorsky, H.~Nakaoka, Y.~Palu, \emph{Positive and negative extensions in extriangulated categories}, arXiv:2103.12482.

\bibitem[HLN]{HLN21}
M.~Herschend, Y.~Liu, H.~Nakaoka, \emph{$n$-exangulated categories (I): Definitions and fundamental properties},  J. Algebra {\bf 570} (2021), 531--586. 

\bibitem[INP]{INP}
O.~Iyama, H.~Nakaoka, Y.~Palu, \emph{Auslander--Reiten theory in extriangulated categories}, arXiv:1805.03776.

\bibitem[KV]{KV88}
B.~Keller, D.~Vossieck, \emph{Aisles in derived categories}, Bull. Soc. Math. Belg. S$\acute{\mathrm{e}}$r. A \textbf{40} (1988), no.~2, 239--253. 

\bibitem[KY]{KY14} 
S.~Koenig, D.~Yang, \emph{Silting objects, simple-minded collections, $t$-structures and co-$t$-structures for finite-dimensional algebras}, Doc. Math. \textbf{19} (2014), 403--438.

\bibitem[Kr]{Kr21}
H.~Krause, \emph{Homological Theory of Representations}, https://www.math.uni-bielefeld.de/~hkrause/HomTheRep.pdf.

\bibitem[LN]{LN19}
Y.~Liu, H.~Nakaoka, \emph{Hearts of twin Cotorsion pairs on extriangulated categories}, J. Algebra {\bf 528} (2019), 96--149.

\bibitem[LZ]{LZ}
Y.~Liu, P.~Zhou, \emph{Hereditary cotorsion pairs on extriangulated subcategories}, arXiv:2012.06997.

\bibitem[MDZ]{MDZ} Y.~Ma, N.~Ding, Y.~Zhang, \emph{Auslander--Buchweitz Approximation Theory for Extriangulated Categories}, arXiv:2006.05112v2.

\bibitem[MLHG]{MLHG}
Y.~Ma, H.~Liu, J.~Hu. Y.~Geng, \emph{A new method to construct model structures from left Frobenius pairs in extriangulated categories}, arXiv:2108.06642v1.

\bibitem[MSSS]{MSSS13}
O.~Mendoza, V.~Santiago, C.~S\'aenz, V.~Souto, \emph{Auslander--Buchweitz context and co-$t$-structures}, Appl. Categ. Structures {\bf 21} (2013), 417--440.

\bibitem[NP]{NP19}
H.~Nakaoka, Y.~Palu, \emph{Extriangulated categories, Hovey twin cotorsion pairs and model structures}, Cah. Topol. G\'eom. Diff\'er. Cat\'eg. {\bf 60} (2019), no.~2, 117--193.

\bibitem[Pa]{Pa08}
D.~Pauksztello, \emph{Compact corigid objects in triangulated categories and co-$t$-structures}, Cent. Eur. J. Math. 6 (2008), no.~1, 25--42.

\bibitem[PZ]{PZ}
D.~Pauksztello, A.~Zvonareva, \emph{Co-t-structures, cotilting and cotorsion pairs}, arXiv:2007.06536.

\bibitem[R]{R91}
C.~M.~Ringel, \emph{The category of modules with good filtrations over a quasi-hereditary algebra has almost split sequences}, Math. Z. {\bf 208} (1991), no.~2, 209--223.

\bibitem[Sa]{S79}
L.~Salce, \emph{Cotorsion theories for abelian groups}, Symposia Mathematica, Vol. XXIII (Conf. Abelian Groups and their
Relationship to the Theory of Modules, INDAM, Rome, 1977) (1979), 11--32, Academic Press, London-New York.

\bibitem[TG]{TG}
L.~Tan, Y.~Gao, \emph{One-sided Frobenius pairs in extriangulated categories}, arXiv:2108.05856.

\bibitem[ZZ]{ZZ20} B.~Zhu, X.~Zhuang, \emph{Tilting subcategories in extriangulated categories}, Front. Math. China, \textbf{15} (2020), 225--253.
\end{thebibliography}
\end{document}